\documentclass[12pt]{article}
\usepackage{etex}
\usepackage{amsmath,amssymb,amsfonts,amsthm}
\usepackage{a4wide}
\usepackage{color}
\usepackage{verbatim}
\usepackage{epsfig,subfigure,epstopdf}
\usepackage[]{graphics}
\usepackage{graphicx,inputenc}
\usepackage{subfigure}
\usepackage[justification=centering]{caption}
\usepackage{pictex}
\usepackage{wrapfig}
\usepackage{multirow}
\usepackage[T1]{fontenc}
\usepackage{authblk}

\usepackage[colorlinks,linktocpage,linkcolor=blue]{hyperref}
\usepackage{fancyhdr}

\newtheorem{theorem}{Theorem}[section]
\newtheorem{lemma}[theorem]{Lemma}
\newtheorem{definition}[theorem]{Definition}

\newtheorem{assumption}[theorem]{Assumption}
\newtheorem{remark}[theorem]{Remark}
\newtheorem{corollary}[theorem]{Corollary}

\numberwithin{equation}{section}

\DeclareMathOperator*{\argmin}{arg\,min}
\def\D{\partial_t^\alpha}

\def\N+{n\in\mathbb{N}^{+}}

\def\A{\mathcal{A}}
\def\l{\langle}
\def\r{\rangle}

\def\S{\text{Span}}

\def\V{\mathbb{V}}
\def\E{\mathbb{E}}
\def\SN{\mathcal{S}_N}
\def\WSN{\widehat{\mathcal{S}}_K}
\def\A{\mathcal{A}}
\def\C{{\rm Cov}}
\def\W{\mathbb{W}}
\def\R{\mathbb{R}}
\def\hun{\widehat{u}_n(T)}
\def\hum{\widehat{u}_m(T)}

\begin{document}

\title{An inverse random source problem in a stochastic fractional diffusion equation}
\author[1]{Pingping Niu\thanks{ppniu14@fudan.edu.cn}}
\author[2]{Tapio Helin\thanks{tapio.helin@helsinki.fi}}
\author[2]{Zhidong Zhang\thanks{zhidong.zhang@helsinki.fi}}
\affil[1]{School of Mathematical Sciences, Fudan University, China}
\affil[2]{Department of Mathematics and Statistics, University of Helsinki, Finland}

\maketitle

\begin{abstract}
 In this work the authors consider an inverse source problem in the following 
 stochastic fractional diffusion equation 
 $$\D u(x,t)+\A u(x,t)=f(x)h(t)+g(x) \dot{\W}(t).$$
 The interested inverse problem is to reconstruct $f(x)$ and $g(x)$ by the 
 statistics of the final time data $u(x,T).$ Some direct problem results are proved 
 at first, such as the existence, uniqueness, representation and regularity of the 
 solution. Then the reconstruction scheme for $f$ and $g$ is given. To tackle the 
 ill-posedness, the Tikhonov regularization is adopted. Finally we give a regularized 
 reconstruction algorithm and some numerical results are displayed.\\\\
 \text{Keywords}: inverse problem, stochastic fractional diffusion equation, 
 random source, Tikhonov regularization, reconstruction, regularity, partial measurements.
\end{abstract}

\section{Introduction}
At a microscopic level, the physical phenomenon of diffusion is related to the random motion of 
individual particles. In one of his celebrated work, Einstein 
\cite{einstein1905molekularkinetischen} deduced that the density function of particles 
satisfies the classical diffusion equation under the 
key assumption that the mean squared displacement over 
a large number of jumps is proportional to time, i.e. $\overline{(\Delta x)^2}\propto t$.
Currently, a large array of physical evidence suggests that there exists also physical diffusion 
that does not satisfy this assumption 
\cite{KlafterSilbey:1980, MetzlerKlafter:2000,BouchaudGeorges:1990,GefenAharonyAlexander:1983}. 
In such \emph{anomalous diffusion} the rate of mean 
squared displacement may satisfy $\overline{(\Delta x)^2}\propto t^\alpha,\ \alpha\ne1.$
The different rate introduces a modification to the diffusion equation in the 
form of time fractional derivative and the corresponding equations 
are often called fractional differential equations (FDEs). The applications of FDEs include, to name a few,
the thermal diffusion in media with fractal geometry  \cite{nigmatullin1986realization}, 
highly heterogeneous aquifer \cite{adams1992field}, 
non-Fickian diffusion in geological formations \cite{berkowitz2006modeling},
mathematical finance \cite{BarkaiMetzlerKlafter:2000}, 
underground environmental problem \cite{hatano1998dispersive} and the analysis on  
viscoelasticity in material science \cite{mainardi2010fractional,wharmby2013generalization,wharmby2014modifying}.

Here we consider an FDE with a random source term
\begin{equation}\label{SFDE}
 \begin{cases}
  \begin{aligned}
   \D u(x,t)+\A u(x,t)&=F(x,t), &&(x,t)\in D\times(0,T],\ \alpha\in(1/2,1);\\  
   u(x,t)&=0, &&(x,t)\in\partial D\times(0,T];\\
   u(x,0)&=0, &&x\in D,
  \end{aligned}
 \end{cases}
\end{equation}
where $\D$ is the Djrbashyan-Caputo fractional derivative given by the expression
\begin{equation*}
	\D u = \frac{1}{\Gamma(1-\alpha)}
	\int_0^t(t-\tau)^{-\alpha}u'(\tau)\,d\tau\quad \text{for}\quad 0<\alpha<1
\end{equation*}
and $\Gamma$ stands for the Gamma function. However, we need a stricter 
restriction $\alpha\in(1/2,1)$ on $\alpha$ for the regularity estimate, 
which can be seen in the proof of Lemma \ref{regularity_L2}. 
Above, $D\subset \mathbb{R}^n$ is open and bounded, and the operator $\A$ with the definition 
\begin{equation*}
 \A u=-\sum_{i,j=1}^n a^{ij}(x)u_{x_ix_j}+\sum_{i=1}^nb^i(x)u_{x_i}+c(x)u
\end{equation*}
with $a^{ij}, b^i, c \in C^\infty(\R^n)$ is symmetric, elliptic and positive definite. The random source term has the expression
\begin{equation*}
	F(x,y) = f(x)h(t)+g(x)\dot{\W}(t),
\end{equation*}
where the function $h\in L^\infty([0,T])$ is known and $\W$ is the 
standard Wiener process on a probability space $(\Omega,\mathcal{F},\mathbb{P})$.
Due to the randomness, we refer to \eqref{SFDE} as stochastic fractional diffusion equation (SFDE) below.
Let us point out that there are also other alternatives
for the definition of the fractional derivative such as the Riemann--Liouville 
formulation, see \cite[Chapter 2.1]{kilbas2006theory}. However, the  
Djrbashyan--Caputo derivative is often preferred due to its convenient properties related to boundary and initial conditions.

In this paper, we study the following inverse problem related to correlation based imaging:
\begin{center}
	given the empirical expectation and correlations of the 
final time data $u(x,T)$, \\ can we recover the unknown functions $f$ and $|g|$?
\end{center}
Notice that the source term $g\dot{\W}$ has an invariant distribution with respect to the sign of $g$. Therefore,
the recovery is considered up to the sign of $g$. We give a positive answer to this question and demonstrate it by numerical simulations.


Correlation based imaging has become common 
in applied inverse problems, where randomness is often an inherent part of the model. 
If the observational data is extensive but exceptionally corrupted or noisy, it can make more sense to
analyze the correlations in the data that connect to the unknown parameters. 
This paradigm has interesting implications to the inverse problems research, since 
first, correlation-based imaging can remarkably reduce the ill-posedness of problems where no analytical solution is known (see \cite{caro2016inverse}) and, second, it introduces a new set of analytical problems that need novel mathematical innovations \cite{garnier2016passive, helin2016correlation}.

Our main contribution in this paper is to demonstrate that 
partial and noisy correlation data under different data acquisition geometries 
can yield useful information regarding the source terms $f$ and $g$.
We come to this conclusion as follows.
We first give a construction of the solution to the stochastic direct problem and give suitable regularity estimates given different a priori smoothness of the source terms.
Based on these results we show stability estimates for recovering $f$ and $|g|$ (Theorem \ref{stability}) and the uniqueness (Theorem \ref{Uniqueness}) given infinite-precision correlation data of the final time solution $u(x,T)$. Meanwhile, the representation and the properties of 
Mittag--Leffler function introduced in section \ref{sec:prelim} yield mild ill-posedness
for the inverse problem (Lemma \ref{ill-posedness}). 
We demonstrate our results in practise with numerical simulations
in section \ref{numerical}. We study different data acquisition geometries 
to find that satisfying localization of the sources can be achieved even
if the observed subdomains are relatively small.

\subsection{Outline of the paper}

This paper is organized as follows. In section \ref{sec:prelim} we collect some preliminary 
material containing the properties of Mittag--Leffler function and the It\^o isometry 
formula, which are crucial in the following proofs. Section \ref{sec:direct} includes 
several results for the forward problem, which support the inverse problem work. 
We study the inverse problem in section \ref{sec:ip}, proving the stability, 
uniqueness and ill-posedness results. Finally, numerical demonstrations are given in section \ref{numerical}.

\subsection{Previous literature}

The fractional differential equations have drawn considerable amount of 
attention among mathematical community lately. Let us mention the work 
by Sakamoto and Yamamoto \cite{sakamoto2011initial} to study the initial 
and boundary value problems for FDEs and the work by Luchko 
\cite{Luchko2009maximum, Luchko2011maximum} to establish the maximum principle 
in FDEs. Moreover, Jin, Lazarov and Zhou \cite{JinLazarovZhou:2016} 
gave a numerical scheme to approximate the FDE by the finite element method. 

In terms of inverse problems, Cheng $et\ al$ 
\cite{cheng2009uniqueness} gave one of the first proofs for a uniqueness theorem in one-dimensional FDE. The article 
\cite{Liu2017reconstruction} considered an inverse source problem in an FDE, which was 
close to this work. The authors in \cite{Li2017analyticity, Rundell2017fractional} 
analyzed the distributed differential equations, in which the assumption 
$\overline{(\Delta x)^2}\propto t^\alpha$ was extended to a more general case 
$\overline{(\Delta x)^2}\propto F(t)$, and studied some inverse problems in such equations. 
For an extensive review of the field we refer to \cite{Jin2015tutorial} and references therein.

Time fractional stochastic PDEs have gained attention recently, see e.g.
\cite{el2002some, sakthivel2012approximate, mijena2015space, zou2016galerkin} and references therein.
Our setup differs slightly from these works: previous studies often assume some spatial randomness of the source, whereas our source term is random only in time. To accommodate the randomness
in the spatial variable, one often smoothens the source in time. This operation is motivated and well explained in \cite{mijena2015space}. Let us also mention that first study of inverse source problems for time fractional stochastic PDEs were carried out in \cite{tuan2017inverse} for discrete random noise.

Correlation based imaging in inverse problems has been considered in applications already for a while,
see e.g. the early work \cite{devaney1979inverse} on inverse random source problems.
Since then correlation based imaging in random source problems has been considered widely 
in the framework of different PDE models by Li, Bao and others \cite{li2017stability, bao2017inverse, li2017inverse, Bao2016several, Bao2014Helmholtz, Li2011source, Bao2010numerical}.
In this regard our paper provides the first study of random source problems in fractional diffusion models. Let us also point out
that correlation based imaging has been considered for problems where the randomness is an inherent property of the medium or boundary condition
\cite{helin2018atmospheric, helin2016correlation, caro2016inverse, garnier2016passive, helin2017inverse, garnier2012correlation, garnier2009passive, borcea2002imaging}.

\section{Preliminaries}
\label{sec:prelim}

Since $\A$ is a symmetric and elliptic operator with domain $L^2_0(D)$, 
then its eigensystem $\{(\lambda_n, \phi_n(x)): \N+\}$ has the following 
properties:   
$0<\lambda_1\le\lambda_2\le\cdots\le\lambda_n<\cdots$ 
and $\{\phi_n:\N+ \}\subset H^2(D)\cap H_0^1(D)$ constitutes an 
orthonormal basis of $L^2(D).$ Throughout the paper, we denote the inner product 
in $L^2(D)$ by $\l\cdot,\cdot\r_{L^2(D)}$. Moreover, we write $f \lesssim g$ for two functions $f,g: X \to \R$ on some domain $X$ if there is a universal constant $C>0$ such that $f(x) \leq C g(x)$ for all parameters $x \in X$. Similarly, we write $f \simeq g$ if both $f \lesssim g$ and $g \lesssim f$ hold.

%

Let us now introduce the Mittag--Leffler function which will play a 
central role in the following analysis. The Mittag--Leffler function is defined as 
\begin{equation*}
E_{\alpha,\beta}(z) = \sum_{k=0}^\infty \frac{z^k}{\Gamma(k\alpha+\beta)}
\end{equation*}
for $z\in \mathbb{C}$.
Notice that this expression generalizes the natural exponential function since 
$E_{1,1}(z)=e^z.$ 

Let us next record some well-known properties of the function $E_{\alpha,\beta}$.
Below, we study the behaviour of $E_{\alpha,\beta}$ only on the negative real line.
However, the statements generalize to the complex plane. For reference, see \cite{Podlubny1999fractional, kilbas2006theory}.

\begin{lemma} \cite[Theorem 1.4]{Podlubny1999fractional}\label{mittag_bound}
Let $0<\alpha<2$ and $\beta\in\mathbb{R}$ be arbitrary. Then it holds that
\begin{equation*}
	\label{eq:ML_decay_rate}
  |E_{\alpha,\beta}(-t)| \leq \frac{C}{1+t}
\end{equation*}
for any $t\geq 0$
and for any $p \in {\mathbb N}$ we have the asymptotic formula
\begin{equation*}
	\label{eq:ML_asymptotic_formula}
  E_{\alpha,\beta}(-t)=-\sum_{k=1}^p \frac{(-t)^{-k}}{\Gamma(\beta-\alpha k)}
  +{\mathcal O}(t^{-1-p})
\end{equation*}
as $t \to \infty$.
\end{lemma}

A useful result related to high order differentials of Mittag--Leffler functions is given by Sakamoto and Yamamoto in \cite{sakamoto2011initial}.

\begin{lemma}\cite[Lemma $3.2$]{sakamoto2011initial}\label{mittag_derivative}
	For $\lambda>0,\ \alpha>0$ and $\N+,$ we have 
	$$\frac{d^n}{dt^n}E_{\alpha,1}(-\lambda t^\alpha)
	=-\lambda t^{\alpha-n}E_{\alpha,\alpha-n+1}(-\lambda t^\alpha),
	\ t>0.$$
\end{lemma}

A function $f: (0,\infty) \to \R$ is called \emph{completely monotonic} if $f\in C^\infty(0,\infty)$ and
\begin{equation*}
	(-1)^n f^{(n)} (t) \geq 0
\end{equation*}
for all $t\in (0,\infty)$, i.e. the derivatives are alternating in sign. For the proof of the following result, see \cite{pollard1948completely} and \cite[Lemma 4.25]{gorenflo2014mittag}.
\begin{lemma}\label{mittag_positive}
\par For $0<\alpha<1,$ functions $t\mapsto E_{\alpha,1}(-t)$ and $t\mapsto E_{\alpha,\alpha}(-t)$ are completely monotonic.
\end{lemma}
Lemma \ref{mittag_positive} yields immediately the next corollary.
\begin{corollary}\label{mittag_positive_1}
	If $0<\alpha<1$ and $t>0,$ then 
	$E_{\alpha,\alpha}(-t)\ge 0.$
\end{corollary}



Finally, let us recall the well-known It\^o isometry formula.

\begin{lemma}[\cite{Bernt2003stochastic}] \label{Ito isometry formula} 
Let $(\Omega,\mathcal{F},\mathbb{P})$ be a probability space and let $f,g : [0,\infty)\times \Omega\rightarrow\mathbb{R}$
satisfy the following properties
\begin{itemize}
 \item[(1)] $(t,\omega)\rightarrow f(t,\omega)$ is $\mathcal{B}\times\mathcal{F}$
 -measurable, where $\mathcal{B}$ denotes the Borel $\sigma$-algebra on 
 $[0,\infty)$;
 \item[(2)] $f(t,\omega)$ is $\mathcal{F}_t$-adapted;
 \item[(3)] $\E \int_S^Tf^2(t,\omega)\mathrm{d}t <\infty$ for some $S,T>0$.
\end{itemize}
Then it follows that
\begin{equation}
	\label{eq:ito_isometry}
 \E\left[\left(\int_S^Tf(t,\omega)~\mathrm{d}\W(t)\right)
 \left(\int_S^Tg(t,\omega)~\mathrm{d}\W(t)\right)\right]
=\E\int_S^Tf(t,\omega)g(t,\omega)~\mathrm{d}t.
\end{equation}
\end{lemma}
Later, we use the identity \eqref{eq:ito_isometry} for non-random functions and, consequently, the expectation on the right-hand side becomes trivial.


\section{Direct problem}
\label{sec:direct}

\subsection{Solution to the SFDE \eqref{SFDE}}

Let us introduce the notion of mild solution for our stochastic fractional differential equation. To make sense of the solution, we need some assumptions regarding the source term.

\begin{assumption}\label{inverse}
We assume that $f,g\in L^2(D)$ such that $g\neq 0$ and $h\in L^\infty(0,T)$ is positive and bounded from below, i.e.,
there exists $c_h>0$ s.t. $h\geq c_h$.
\end{assumption}

\begin{definition}\label{mild_solution}
A stochastic process $u: [0,T] \times \Omega \to L^2(D)$ defined by
 \begin{equation}\label{eq:mild_solution}
   u(\cdot,t,\omega)= \sum_{n=1}^\infty (I_{n,1}(t)+I_{n,2}(t,\omega)) \phi_n(\cdot),
 \end{equation}
where   
\begin{equation*}
\begin{aligned}
	I_{n,1}(t) &= f_n\int_0^t (t-\tau)^{\alpha-1} E_{\alpha,\alpha}(-\lambda_n(t-\tau)^\alpha) h(\tau)\ d\tau,\\   
	I_{n,2}(t,\omega) &= g_n \int_0^t (t-\tau)^{\alpha-1} E_{\alpha,\alpha}(-\lambda_n(t-\tau)^\alpha)\ d\W(\tau),
\end{aligned}
\end{equation*}
with
\begin{equation*}
 f_n=\l f(\cdot),\phi_n(\cdot)\r_{L^2(D)},\ g_n=\l g(\cdot),\phi_n(\cdot)\r_{L^2(D)},
\end{equation*}
is called a \emph{mild solution} of equation \eqref{SFDE}.
\end{definition}





The regularity of \eqref{eq:mild_solution} is proved below in Lemma \ref{regularity_L2}.
Notice also that the term $I_{n,1}(t)$ is fully deterministic and contains only information regarding the deterministic part of the source. Similarly, the term $I_{n,2}$ carries the information related to the stochastic source. In the following, we omit the notation $\omega$ for brevity and make the connection to the random element implicit.

\begin{remark}
The mild solutions to more general time-fractional stochastic PDEs have been considered in \cite{sakthivel2012approximate,zou2016galerkin} based on the semigroup approach taken in \cite{el2002some}. Our construction is related but uses the approach introduced by Sakamoto and Yamamoto in \cite{sakamoto2011initial}.
\end{remark}

\begin{lemma}\label{regularity_L2}
The stochastic process $u$ given in \eqref{eq:mild_solution} satisfies
 \begin{equation*}\label{regularity_1}
 \E\|u\|^2_{L^2(D\times[0,T])} \le C\Big(\|h\|_{L^2(0,T)}^2 \|f\|_{L^2(D)}^2
 +T^{2\alpha}\|g\|_{L^2(D)}^2\Big),
\end{equation*}
where $C>0$ is a constant.
\end{lemma}
\begin{proof}
 Recall that $\{\phi_n:n \in {\mathbb N}\}$ is an orthonormal basis of $L^2(D)$.  
Now for each $t\in [0,T]$ it holds that  
\begin{equation*}
 \begin{aligned}
  \|u(\cdot,t)\|^2_{L^2(D)}&
  =\left\|\sum_{n=1}^\infty (I_{n,1}(t)+ I_{n,2}(t))\phi_n(\cdot)\right\|^2_{L^2(D)}
  =\sum_{n=1}^\infty (I_{n,1}(t)+ I_{n,2}(t))^2\\
  &\le 2\sum_{n=1}^\infty [I_{n,1}(t)]^2
  +2\sum_{n=1}^\infty [I_{n,2}(t)]^2.
 \end{aligned}
\end{equation*}
Hence we have
\begin{equation*}
 \begin{aligned}
  \E \|u\|^2_{L^2(D\times[0,T])}
  &=\E\left[\int_0^T \|u(\cdot,t)\|^2_{L^2(D)}\ dt\right]\\
  &\lesssim \E \left[\int_0^T \sum_{n=1}^\infty [I_{n,1}(t)]^2
  +\sum_{n=1}^\infty [I_{n,2}(t)]^2\ dt\right]\\
  &=\int_0^T \sum_{n=1}^\infty [I_{n,1}(t)]^2 \ dt
  + \E \left[ \int_0^T\sum_{n=1}^\infty [I_{n,2}(t)]^2\ dt\right]\\
  &=\sum_{n=1}^\infty \| I_{n,1}\|_{L^2(0,T)}^2
  +\int_0^T\sum_{n=1}^\infty  \E I_{n,2}^2(t) \ dt\\
  &:=S_1+S_2.
 \end{aligned}
\end{equation*}

First, consider the sum $S_1$. We can write the term $I_{n,1}$ as the convolution
\begin{equation*}
	I_{n,1}(t)= f_n (G_{\alpha,n}* h)(t)
\end{equation*} 
where
\begin{equation*}
	G_{\alpha,n}(t) = t^{\alpha-1}E_{\alpha,\alpha}(-\lambda_nt^\alpha)
\end{equation*}
and, therefore, the Young's convolution inequality yields 
\begin{equation*}
 \|I_{n,1}\|_{L^2(0,T)}
 \le f_n \|G_{\alpha,n}\|_{L^1(0,T)}\cdot
 \| h\|_{L^2(0,T)};
\end{equation*}
while the following result is derived from Lemmas \ref{mittag_derivative}, 
\ref{mittag_positive} and Corollary \ref{mittag_positive_1} 
\begin{equation*}
 \|G_{\alpha,n}\|_{L^1(0,T)}
 =\int_0^T t^{\alpha-1}E_{\alpha,\alpha}(-\lambda_nt^\alpha) dt
 =\frac{1-E_{\alpha,1}(-\lambda_nT^\alpha)}{\lambda_n}\le \frac{1}{\lambda_1}.
\end{equation*}
In consequence, we can find the upper bound for $S_1$ as follows
\begin{equation*}
 S_1 \leq
 \frac 1{\lambda^2_1} \|h\|_{L^2(0,T)}^2 \sum_{n=1}^\infty f_n^2
 =C\|h\|_{L^2(0,T)}^2 \|f\|_{L^2(D)}^2.
\end{equation*}

Second, let us consider $S_2$. For any $t\in [0,T]$ we have 
\begin{equation*}
 \E I^2_{n,2}(t) =g_n^2\int_0^t\tau^{2\alpha-2}
 [E_{\alpha,\alpha}(-\lambda_n\tau^\alpha)]^2\ d\tau
 \le g_n^2\int_0^t \tau^{2\alpha-2} C^2 \ d\tau
 =Cg_n^2t^{2\alpha-1},
\end{equation*}
where we applied Lemmas \ref{mittag_bound}, \ref{Ito isometry formula} 
and the restriction $\alpha\in(1/2,1)$. Thus, the estimate of $S_2$ can be bounded by 
\begin{equation*}
  S_2\le \int_0^T\sum_{n=1}^\infty Cg_n^2t^{2\alpha-1}\ dt
  =CT^{2\alpha} \sum_{n=1}^\infty g_n^2
  \le CT^{2\alpha}\|g\|_{L^2(D)}^2.
\end{equation*}
Finally, combining the estimates for $S_1$ and $S_2$  
yields the desired result.
\end{proof}

Lemma \ref{regularity_L2} considered the $L^2$ 
regularity of the solution over time and space. However, one can also study the 
space $L^2$-bound for $u$ at a given time $t$.

\begin{lemma}\label{regularity_infty}
The supremum of the expected norm of the solution satisfies
\begin{equation*}
	\label{regularity_3}
	\sup_{0\le t\le T}\E\left[\|u(\cdot,t)\|^2_{L^2(D)}\right] \le C\Big(\|h\|_{L^\infty[0,T]}^2\|f\|^2_{L^2(D)}
 +T^{2\alpha-1}\|g\|^2_{L^2(D)}\Big).
\end{equation*} 
Moreover, if one has in addition that $g \in H^2(D)$, then
\begin{equation*}
 \sup_{0\le t\le T}\E\left[\|u(\cdot,t)\|^2_{H^2(D)} \right]
 \le C\Big(\|h\|_{L^\infty[0,T]}^2\|f\|^2_{L^2(D)} +
 T^{2\alpha-1}\|g\|^2_{H^2(D)}\Big).
\end{equation*}
\end{lemma}
\begin{proof}
 From the proof of Lemma \ref{regularity_L2} we conclude that
 \begin{equation*}
  \|u(\cdot,t)\|^2_{L^2(D)}\lesssim \sum_{n=1}^\infty [I_{n,1}(t)]^2
  +\sum_{n=1}^\infty [I_{n,2}(t)]^2
  \end{equation*}
  and
  \begin{equation*}
   \|u(\cdot,t)\|^2_{H^2(D)} \simeq \|{\mathcal A}u(\cdot,t)\|^2_{L^2(D)}\lesssim \sum_{n=1}^\infty \lambda_n^2 [I_{n,1}(t)]^2
  +\sum_{n=1}^\infty \lambda_n^2 [I_{n,2}(t)]^2.
 \end{equation*}
Similar to the proof of the previous lemma, we have
\begin{equation*}
 |I_{n,1}(t)|= | f_n (G_{\alpha,n}* h)(t)|
\le |f_n|\cdot\|h\|_{L^\infty[0,T]}\int_0^t G_{\alpha,n}(\tau) d\tau
\le \frac 1{\lambda_n}|f_n|\cdot\|h\|_{L^\infty[0,T]}
\end{equation*}
and
 \begin{equation*}
   \E I^2_{n,2}(t) \le Cg_n^2t^{2\alpha-1}.
 \end{equation*}
Hence, we can deduce that 
\begin{eqnarray*}
 \sup_{0\le t\le T}\E\left[\|u(\cdot,t)\|^2_{L^2(D)}\right]
 & \lesssim & \Big(\|h\|_{L^\infty[0,T]}^2\sum_{n=1}^\infty f_n^2+T^{2\alpha-1}
 \sum_{n=1}^\infty g_n^2\Big)\\
 & =& \Big(\|h\|_{L^\infty[0,T]}^2\|f\|^2_{L^2(D)}+T^{2\alpha-1}\|g\|^2_{L^2(D)}\Big)
\end{eqnarray*}
and
\begin{eqnarray*}
  \sup_{0\le t\le T}\E\left[\|u(\cdot,t)\|^2_{H^2(D)}\right]
  & \lesssim & \Big(\|h\|_{L^\infty[0,T]}^2\sum_{n=1}^\infty f_n^2+T^{2\alpha-1}
 \sum_{n=1}^\infty \lambda_n^2g_n^2\Big)\\
 & \lesssim & \Big(\|h\|_{L^\infty[0,T]}^2\|f\|^2_{L^2(D)}+
 T^{2\alpha-1}\|g\|^2_{H^2(D)}\Big).
\end{eqnarray*}
This completes the proof.
\end{proof}

\section{Reconstruction of $f$ and $|g|$ from the final time correlations}
\label{sec:ip}

In this section we consider the inverse problem of reconstructing $f$ and $|g|$.
Let $X,Y : \Omega \to \R$ be random variables on some complete probability space. 
Below, we write
$\V (X) := \E (X-\E X)^2$
and
\begin{equation*}
 \C(X,Y):=\E(X-\E X)(Y-\E Y )
\end{equation*}
for the variance and covariance, respectively.
We assume that our data is partial information regarding 
the distributions of random variables $u_n(T)$ defined by
$$u_n(T):=\l u(\cdot,T),\phi_n(\cdot)\r_{L^2(D)}$$ 
for any $\N+$.

\subsection{Stability of the reconstruction}
From Definition \ref{mild_solution} and Lemma \ref{Ito isometry formula} it follows that
the final time expectation and variance can be formulated as
 \begin{equation}\label{EV2}
 \begin{aligned}
 \E u_n(T) &=f_n\int_0^T \tau^{\alpha-1} 
 E_{\alpha,\alpha}(-\lambda_n\tau^\alpha) h(T-\tau)\ d\tau,\\
 \V (u_n(T))&=g_n^2\int_0^T \tau^{2\alpha-2} 
 [E_{\alpha,\alpha}(-\lambda_n\tau^\alpha)]^2\ d\tau
 \end{aligned}
\end{equation}
for any $\N+$. To show the stability result, we deduce the coming lemma at first.
\begin{lemma}\label{estimate_ge}
For each $\N+$, there exists a constant $C>0$ independent of $n$ such that  
 \begin{equation*}
   \left|\int_0^T \tau^{\alpha-1}E_{\alpha,\alpha}(-\lambda_n\tau^\alpha)
   h(T-\tau)\ d\tau\right| \ge C\lambda_n^{-1}
 \end{equation*}
 and
 \begin{equation*}
   \left|\int_0^T \tau^{2\alpha-2}[E_{\alpha,\alpha}(-\lambda_n\tau^\alpha)]^2
   \ d\tau\right| \ge C\lambda_n^{-2}.
 \end{equation*}
\end{lemma}
\begin{proof}
For the first estimate, by Lemma \ref{mittag_derivative} and Assumption \ref{inverse}, 
we obtain
\begin{equation*}
\begin{aligned}
 \left|\int_0^T \tau^{\alpha-1} E_{\alpha,\alpha}
  (-\lambda_n\tau^\alpha) h(T-\tau)\ d\tau\right|
  &\ge c_h \int_0^T \tau^{\alpha-1} E_{\alpha,\alpha}
  (-\lambda_n\tau^\alpha)\ d\tau\\  
 &=c_h \lambda_n^{-1} [1-E_{\alpha,1}(-\lambda_nT^\alpha)]\\
  &\ge c_h \lambda_n^{-1} [1-E_{\alpha,1}(-\lambda_1T^\alpha)]\\
  &   \gtrsim \lambda_n^{-1}.
\end{aligned} 
\end{equation*}
For the second one, Lemmas \ref{mittag_bound} and \ref{mittag_derivative} yield 
\begin{equation*}
 \begin{aligned}
  \left|\int_0^T \tau^{2\alpha-2}[E_{\alpha,\alpha}(-\lambda_n\tau^\alpha)]^2
   \ d\tau\right|
  &\ge [E_{\alpha,\alpha}(-\lambda_nT^\alpha)]^2\int_0^T \tau^{2\alpha-2}
   \ d\tau\\
  &\gtrsim T^{2\alpha-1} \lambda_n^{-2}T^{-2\alpha} \\
  &\gtrsim \lambda_n^{-2}
 \end{aligned}
\end{equation*}
and complete the proof.
\end{proof}

Now a stability result follows in a straightforward manner.
\begin{theorem}[Stability]\label{stability}
Suppose Assumption \ref{inverse} is satisfied and, in addition, $g\in H^2(D)$.
Then there exists a constant $C>0$ such that
 \begin{equation*}
 \|f\|_{L^2(D)}^2 + \|g\|_{L^2(D)}^2 \le C\E\|u(\cdot,T)\|_{H^2(D)}^2.
\end{equation*}
\end{theorem}
\begin{proof}
 Lemma \ref{estimate_ge} and the Jensen inequality yield that 
 \begin{equation*}
 f_n^2 + g_n^2 \lesssim \lambda_n^2 \left(\left(\E u_n(T)\right)^2 + \V (u_n(T))\right) =  \lambda_n^2
 \E (u_n(T))^2.
 \end{equation*}
Therefore, it follows that 
\begin{equation*}
\|f\|_{L^2(D)}^2 + \|g\|_{L^2(D)}^2 = \sum_{n=1}^\infty (f_n^2+ g_n^2) \lesssim \sum_{n=1}^\infty
\lambda_n^2
 \E (u_n(T))^2 \lesssim \E\|u(\cdot,T)\|_{H^2(D)}^2.
\end{equation*} 
The proof is complete.
\end{proof}

\subsection{Uniqueness of the reconstruction}

As discussed above, the stochastic FDE in \eqref{SFDE} is invariant with respect to the sign of $g$. Therefore, the observations of the final time do not contain information regarding the sign.
However, notice carefully that the observed expectation and variance do not ensure uniqueness for $|g|$, since each process $\langle u, \phi_n\rangle_{L^2(D)}$ is invariant to the sign of $g_n$ independently. As we will see below, the cross-covariance between $u_n(T)$ and $u_k(T)$ for $k\neq n$ adds the crucial information to the system since the random white noise in \eqref{SFDE} is only time-dependent.

By our Definition \ref{mild_solution} and Lemma \ref{Ito isometry formula}, the covariance 
$\C(u_{m}(T),u_n(T))$ is given by identity
\begin{equation}
	\label{cov_expression}
 \C(u_{m}(T),u_n(T))
=g_{m}g_n\int_0^T \tau^{2\alpha-2} E_{\alpha,\alpha}(-\lambda_{m}\tau^\alpha)
 E_{\alpha,\alpha}(-\lambda_n\tau^\alpha)\ d\tau
\end{equation}
for any $m,n\in {\mathbb N}^+$.
A uniqueness result can now be provided as follows.
\begin{theorem}[Uniqueness]\label{Uniqueness}
 Suppose Assumption \ref{inverse} holds and $g \in H^2(D)$. Moreover, 
 let $N_0$ be an index such that $\langle g, \phi_{N_0}\rangle_{L^2(D)} \neq 0$. 
 The expectation of the final time solution and the correlations at $N_0$, i.e.  
 the quantities
 \begin{equation*}
 	\label{fulldata}
 \left \{\E [u_n(T)],  \V (u_{N_0}(T)), \C(u_{N_0}(T), u_n(T)) \; : \; \N+\right\}
 \end{equation*}
determine the source terms $f$ and $|g|$ uniquely.
\end{theorem}
\begin{proof}
First, we clearly have
\begin{equation*}\label{f_n}
  f_n =\frac{\E u_n(T)}{\int_0^T \tau^{\alpha-1} 
 E_{\alpha,\alpha}(-\lambda_n\tau^\alpha) h(T-\tau)\ d\tau},
\end{equation*}
which is well-defined due to Assumption \ref{inverse}.

Second, due to the assumption on $N_0$, the variance $\V (u_{N_0}(T))$ yields $|g_{N_0}|$ up to the sign from equation \eqref{EV2}. For convenience, we pick the positive solution of $g_{N_0}.$
It follows that
\begin{equation*}\label{g_n}
  g_n =\frac{\C(u_{N_0}(T), u_n(T))}{g_{N_0}\int_0^T \tau^{2\alpha-2} 
 E_{\alpha,\alpha}(-\lambda_{N_0}\tau^\alpha)E_{\alpha,\alpha}(-\lambda_n\tau^\alpha)
 \ d\tau}.
\end{equation*}
The integral in the denominator is strictly positive due to $g_{N_0}>0$. 
\end{proof}

Ill-posedness of the recovery can be characterized by the following lemma.
\begin{lemma}\label{ill-posedness}
There exists a constant $C$ which is independent of $n$ such that  
 \begin{equation*}
   \left|\int_0^T \tau^{\alpha-1}E_{\alpha,\alpha}(-\lambda_n\tau^\alpha)
   h(T-\tau)\ d\tau\right| \le C\lambda_n^{-1}
   \end{equation*}
   and
   \begin{equation*}
   \left|\int_0^T \tau^{2\alpha-2}  E_{\alpha,\alpha}(-\lambda_{N_0}\tau^\alpha)
   E_{\alpha,\alpha}(-\lambda_n\tau^\alpha)\ d\tau\right| 
   \le C\lambda_n^{-1+\frac{1}{2\alpha}}.
 \end{equation*}
\end{lemma}
\begin{proof}
Lemma \ref{mittag_derivative} implies that
 \begin{eqnarray*}
  \left|  \int_0^T \tau^{\alpha-1} E_{\alpha,\alpha}
  (-\lambda_n\tau^\alpha) h(T-\tau)\ d\tau\right|
  &\le & \|h\|_{L^\infty(0,T)} \int_0^T \tau^{\alpha-1} E_{\alpha,\alpha}
  (-\lambda_n\tau^\alpha)\ d\tau\\
  &=& \|h\|_{L^\infty(0,T)} \lambda_n^{-1} (1-E_{\alpha,1}(-\lambda_nT^\alpha)) \\
  &\lesssim &\lambda_n^{-1}.
 \end{eqnarray*}
On the other hand, if we let 
\begin{equation}
\label{tstar}
t_*=\lambda_n^{-\frac{1}{2\alpha}}
\end{equation}
and use Lemma \ref{mittag_bound}
to obtain
\begin{equation*}
	E_{\alpha,\alpha}(-\lambda_n t^{\alpha}) \leq 
	\begin{cases}
	C & {\rm for} \; t<t_* \;{\rm and} \\
	\frac{C}{\lambda_n t^\alpha} & {\rm for} \; t\geq t_*,
	\end{cases}
\end{equation*}
then it follows that
\begin{align*}
  \Bigg|\int_0^T \tau^{2\alpha-2}  E_{\alpha,\alpha}(-\lambda_{N_0}\tau^\alpha)
 & E_{\alpha,\alpha}(-\lambda_n\tau^\alpha)\ d\tau\Bigg| \\
 & =\int_0^{t_*} \tau^{2\alpha-2} E_{\alpha,\alpha}(-\lambda_{N_0}\tau^\alpha)
 E_{\alpha,\alpha}(-\lambda_n\tau^\alpha) \ d\tau \\
 &  \quad+\int_{t_*}^T \tau^{2\alpha-2} E_{\alpha,\alpha}(-\lambda_{N_0}\tau^\alpha)
 E_{\alpha,\alpha}(-\lambda_n\tau^\alpha)\ d\tau\\
 & \le \int_0^{t_*} \tau^{2\alpha-2} C^2\ d\tau 
 +\int_{t_*}^T \tau^{2\alpha-2}C\tau^{-\alpha}(\lambda_n\tau^\alpha)^{-1}\ d\tau\\
 &\lesssim  t_*^{2\alpha-1}+\lambda_n^{-1}t_*^{-1}
 -\lambda_n^{-1}T^{-1} \\
 & \lesssim \lambda_n^{-1+\frac{1}{2\alpha}}.
\end{align*}
Above, we find that the choice in \eqref{tstar} optimizes the rate. This completes the proof.
\end{proof}

\section{Numerical reconstruction}\label{numerical}
In this section we illustrate the practical solvability of the inverse problem 
by numerical demonstrations. We consider to reconstruct $f$ and $|g|$ 
in the finite dimensional space 
\begin{equation*}
 \SN:=\S\{\phi_n:1\le n\le N\},
\end{equation*}
where $\phi_n$ are the eigenfunctions of ${\mathcal A}$, 
and denote the approximations of $f$ and $g$ as   
\begin{equation*}
f_N(x)=\sum_{n=1}^N f_n\phi_n(x),\quad g_N(x)=\sum_{n=1}^N g_n \phi_n(x).
\end{equation*}
Also the vector formulations of $f_N$ and $g_N$ can be given as  
\begin{equation*}
 \vec{f}_N=
 \begin{bmatrix}
  f_1&f_2&\cdots& f_N
 \end{bmatrix},
 \quad
 \vec{g}_N=
 \begin{bmatrix}
  g_1&g_2&\cdots& g_N
 \end{bmatrix}.
\end{equation*}

The domain $D$ is set to be the unit circle in $\mathbb{R}^2$ and we let 
$\A=-\Delta$, then it follows that the eigenfunctions 
of $\A$ are given by
\begin{equation*}
\phi_n(r,\theta)=w_nJ_m(\sqrt{\lambda_n}r)\cos{(m\theta+d_n)}, 
\end{equation*} 
where $(r,\theta)$ are the polar coordinates on $D$, 
the phase $d_n$ is either $0$ or $\pi/2,$ $w_n$ is the 
normalized weight factor and $J_m(z)$ is the first kind Bessel function with degree $m$. 
The eigenvalues $\{\lambda_n:\N+\}$ are the squares of the zeros of the 
class of Bessel functions $\{J_m(z):m\in\mathbb{N}\}$ and indexed by $n$ with 
nondecreasing order. Hence, we can see the index $m$ is a function of $n,$ 
i.e. $m=m(n).$ The set $\{\lambda_n:\N+\}$ can be solved numerically and 
satisfy $\lambda_j \simeq j^2.$
The data used in all examples below is simulated and the forward solver being used is based on a 
finite difference scheme. We run the forward solver $10^3$ times for different realizations of the source term and 
average the final time data $u(x,T)$ to get the approximatively exact data 
$\widehat{{\bf E}},\widehat{{\bf C}}$. Lastly, we generate the noisy data 
$\widehat{{\bf E}}^\delta,\widehat{{\bf C}}^\delta$ for all examples by adding
 $1\%$ relative noise. 

We consider the two experiments $(e1)$ and $(e2)$, where
we use the following source terms:
\begin{equation*}
 \begin{aligned}
  (e1):\quad f(r,\theta)=&10w_1J_{m(1)}(\sqrt{\lambda_1}r)\cos{(m(1)\theta)}+
  5w_2J_{m(2)}(\sqrt{\lambda_2}r)\cos{(m(2)\theta)}\\
  &+12w_2J_{m(2)}(\sqrt{\lambda_2}r)\sin{(m(2)\theta)},\\
  g(r,\theta)=&10w_1J_{m(1)}(\sqrt{\lambda_1}r)\cos{(m(1)\theta)}+
  2w_2J_{m(2)}(\sqrt{\lambda_2}r)\cos{(m(2)\theta)}\\
  &+13w_2J_{m(2)}(\sqrt{\lambda_2}r)\sin{(m(2)\theta)};\\
  (e2):\quad f(x,y)=&6\chi_{_{[(x-0.3)^2+0.5(y-0.2)^2<0.2^2]}},\\
  g(x,y)=&-3\chi_{_{[0.3(x+0.4)^2+(y+0.3)^2<0.15^2]}}.
 \end{aligned}
\end{equation*}
The source terms in $(e1)$ and $(e2)$ are represented in Figure \ref{exact}.
\begin{figure}[th!]
	\center
\begin{subfigure}
  \centering
\includegraphics[trim = .5cm .5cm .5cm .8cm, clip=true,height=6.2cm,width=16.2cm]
		{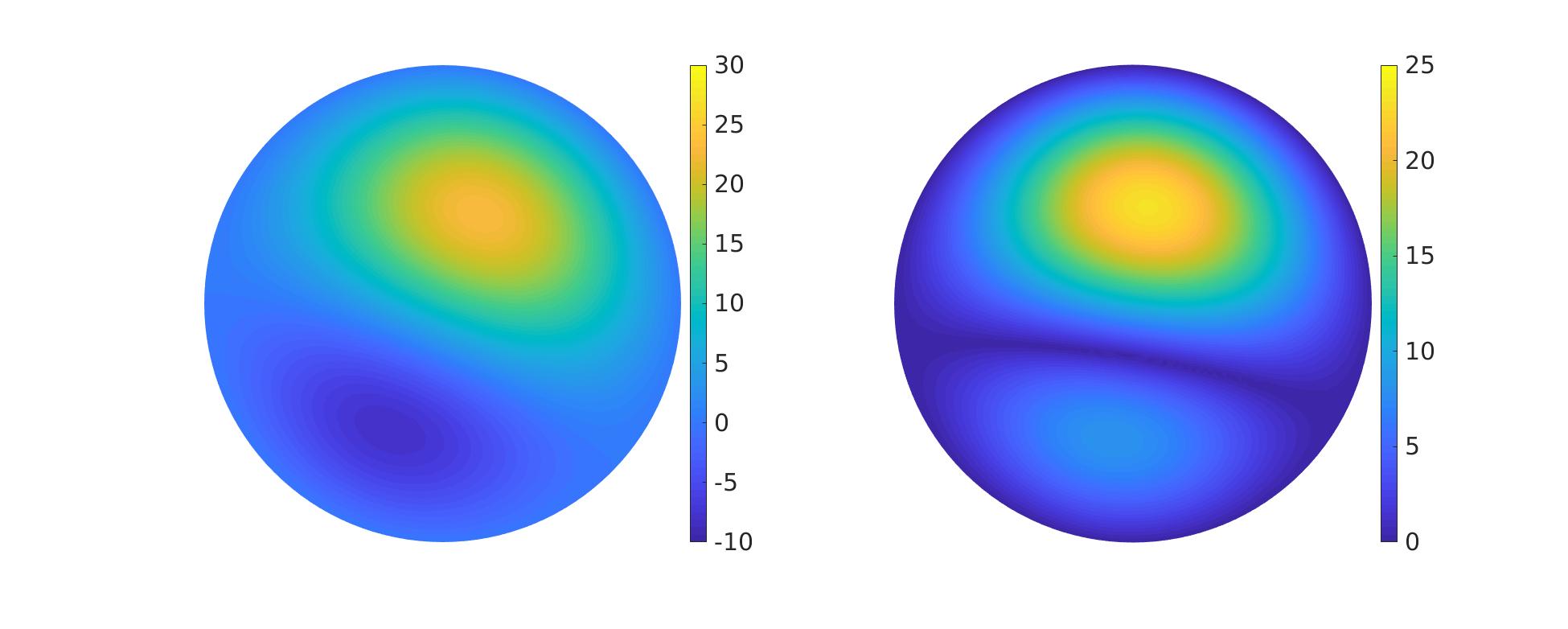}
\end{subfigure}
\\
\begin{subfigure}
  \centering
\includegraphics[trim = .5cm .5cm .5cm .8cm, clip=true,height=6.2cm,width=16.2cm]
		{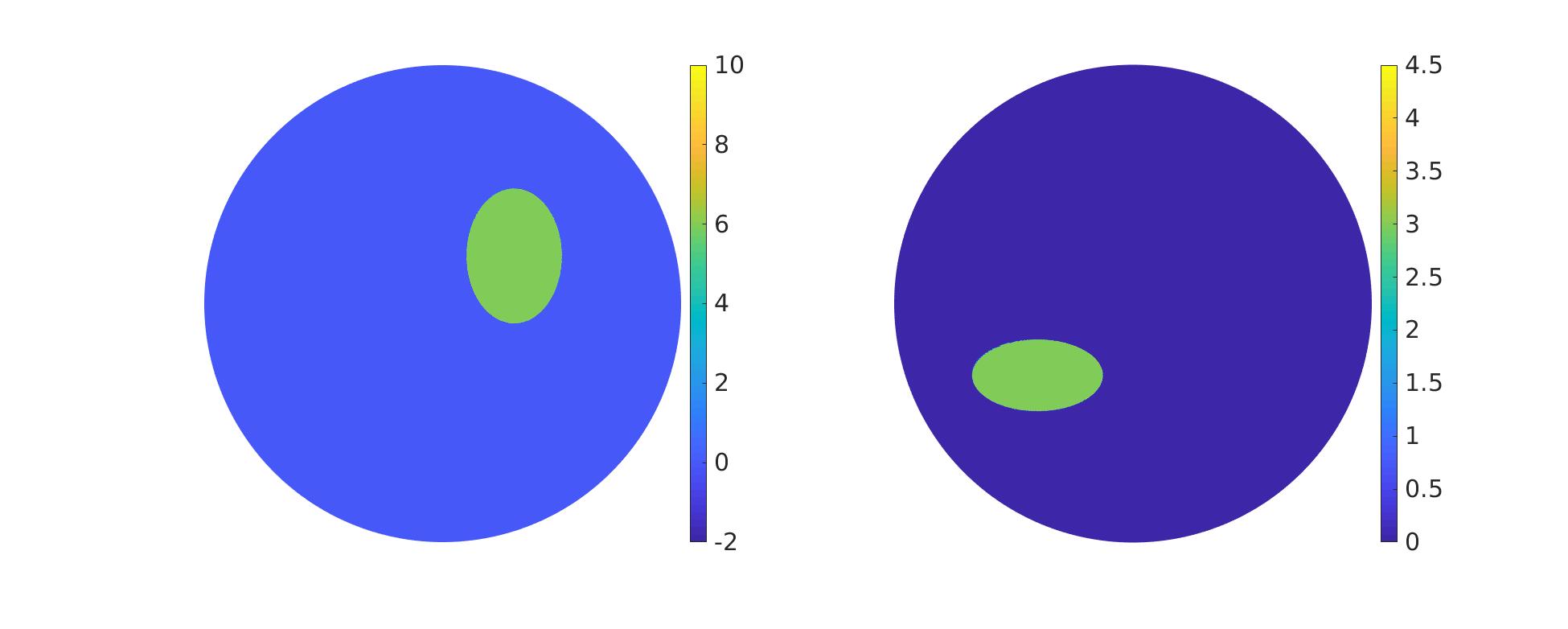}
\end{subfigure}
\caption{\small Exact solutions of $(e1)$ (top) and $(e2)$ (bottom): 
$f$ (left), $|g|$ (right).}
\label{exact}
\end{figure}

\subsection{Data acquisition and finite-dimensional data}

In practise, the data acquisition is unlikely to happen in the basis
$\phi_n$ indicated by ${\mathcal A}$.
For example, the fact that functions $\phi_n$ are not local can be restrictive, 
if the observations are limited to a strict subset 
$D_{mea}\subset D$. To accommodate this thought, suppose our data is given on 
the basis functions of a finite dimensional subspace $\WSN \subset L^2(D)$ such that
\begin{equation*}
 \WSN=\S\{\psi_n(x):n=1,... ,K\},
\end{equation*}
and our data is given by
 \begin{equation*}
 	\label{partialdata}
  \{\E \hun ,  \C(\widehat{u}_k(T), \widehat{u}_\ell(T)) \; : \; 
  k\in {\mathcal I}, \ell \in {\mathcal J}\; {\rm and} \; 
  n \in {\mathcal I}\cup {\mathcal J}\}
 \end{equation*}
 where $\hun = \langle u(T), \psi_n\rangle_{L^2(D)}$ and 
 ${\mathcal I},{\mathcal J} \subset \{1,...,K\}$ are some index subsets. 
  For convenience, we assume that ${\mathcal I} ={\mathcal J} = \{1,...,K\}$ 
  and, therefore, omit denoting the dependence on the index sets. \\

%
%
%


{\bf Source-to-expectation mapping.} Writing $\E \widehat{u}_n(T)$ in the $\{\phi_k\}_{k=1}^\infty$ basis yields
\begin{eqnarray*}
 \E \widehat{u}_n(T) &= &  \sum_{k=1}^\infty \langle \psi_n, \phi_k\rangle  \E \langle u(T), \phi_k \rangle \\
 & = &  \sum_{k=1}^\infty \langle \psi_n, \phi_k\rangle \cdot  f_k\int_0^T \tau^{\alpha-1}  E_{\alpha,\alpha}(-\lambda_k\tau^\alpha) h(T-\tau)\ d\tau.  
 \end{eqnarray*}
Therefore, by using notation $\widehat{{\bf E}} = (\E \widehat{u}_n(T))_{n=1}^K \in \R^K$, we have identity
\begin{equation*}
	\widehat{{\bf E}} = A f,
\end{equation*}
where the operator $A: L^2(D) \to \R^N$ is linear and bounded due to Lemma \ref{regularity_infty} and satisfies
\begin{equation*}
	(A f)_n = \langle z_n, f\rangle
\end{equation*}
with 
\begin{equation*}
	z_n =  \sum_{k=1}^\infty \int_0^T \tau^{\alpha-1}  E_{\alpha,\alpha}(-\lambda_k\tau^\alpha) h(T-\tau)\ d\tau \cdot  \langle \psi_n, \phi_k\rangle \phi_k.
\end{equation*}
\\

{\bf Source-to-covariance mapping.}
We see that we have
\begin{eqnarray*}
	\C(\hum, \hun) & = &  \sum_{k=1}^\infty \sum_{\ell=1}^\infty \C(u_k, u_\ell)  \langle \psi_m, \phi_k \rangle \langle \psi_n, \phi_\ell \rangle \\
	& = & \boldsymbol \psi_m^\top {\bf C} \boldsymbol \psi_n,
\end{eqnarray*}
where $\boldsymbol \psi_m = (\langle \psi_m, \phi_k \rangle)_{k=1}^\infty$ and 
${\bf C} = (\C(u_k, u_\ell))_{k,\ell=1}^\infty$. By writing ${\bf R} = (\boldsymbol \psi_1, ..., \boldsymbol \psi_K)$, we see
\begin{equation*}
	\widehat{{\bf C}} = {\bf R}^\top {\bf C} {\bf R}.
\end{equation*}
Recall now the expression for $\C(u_{m}(T),u_n(T))$ in \eqref{cov_expression}.  We can rewrite \eqref{cov_expression} in the form
 \begin{equation*}
 	{\bf C} = \int_0^T {\bf g}(\tau) {\bf g}(\tau)^\top d\tau,
 \end{equation*}
where 
\begin{equation*}
	{\bf g}(\tau) = \left(g_k \tau^{\alpha-1} E_{\alpha,\alpha}(-\lambda_{k}\tau^\alpha)\right)_{k=1}^\infty : [0,T] \to \R^\infty.
\end{equation*}
 In consequence, we have
 \begin{equation}
 	\label{Chat_aux1}
 	\widehat{{\bf C}} =  \int_0^T {\bf R}^\top {\bf g}(\tau) {\bf g}(\tau)^\top {\bf R} \ d\tau.
 \end{equation}
 Let us consider now the integrand in \eqref{Chat_aux1}. We obtain
 \begin{eqnarray*}
 	({\bf R}^\top {\bf g}(\tau))_m & = & {\boldsymbol \psi_m} \cdot {\bf g}(\tau) \\
 	& = & \sum_{k=1}^\infty  \langle \psi_m, \phi_k \rangle 
 	\langle g, \phi_k \rangle \tau^{\alpha-1} 
 	E_{\alpha,\alpha}(-\lambda_{k}\tau^\alpha)\\
 	& =&  \left \langle g, \sum_{k=1}^\infty\langle \psi_m, \phi_k \rangle  
 	\tau^{\alpha-1} E_{\alpha,\alpha}(-\lambda_{k}\tau^\alpha)\phi_k\right\rangle \\
 	& = & \langle g, w_{m}(\tau) \rangle
 \end{eqnarray*}
 where
 \begin{equation*}
 	w_m(\tau) = \sum_{k=1}^\infty\langle \psi_m, \phi_k \rangle  \tau^{\alpha-1} 
 	E_{\alpha,\alpha}(-\lambda_{k}\tau^\alpha)\phi_k.
 \end{equation*}
Let us define an operator $B : H^2(D) \to \R^{K\times K}$ by
\begin{equation}
	\label{eq:operatorB}
	Bg = \int_0^T {\bf R}^\top {\bf g}(\tau) {\bf g}(\tau)^\top {\bf R}\ d\tau.
\end{equation} 
Clearly, due to Lemma \ref{regularity_infty} the operator $B$ is bounded.
Now we can state the discretized equations for $f$ and $g$: 
\begin{equation*}\label{discretization}
	Af=\widehat{{\bf E}}\quad {\rm and} \quad Bg=\widehat{{\bf C}}. 
\end{equation*}

\subsection{Numerical results with observations on the full domain}\label{fully} 

Here we investigate the numerical reconstruction with  
observations on the full domain, i.e. $D_{mea}=D$, but with correlations
based on one fixed point. In other words, we assume that 
 $\{\psi_n\}$ coincide with $\{\phi_n\}$, ${\mathcal J} = \{1,...,N\}$ 
 and $\mathcal{I}=\{N_0\}$ where $N_0$ is such that 
 $\langle g, \phi_{N_0}\rangle_{L^2(D)} \neq 0$. Moreover,  
 what is interesting, this formulation leads to a linear interpretation 
 of the operator $B$ in \eqref{eq:operatorB}.


The parameters used in these experiments are set as  
\begin{equation}\label{parameter}
 \alpha=0.8,\ T=1,\ h(t)\equiv1,\ N=36,\ N_0=1.
\end{equation}
The numerical results are displayed in Figures \ref{e1} and \ref{e2}, which show  
that the method localizes the sources well. The relative $L^2$ errors are 
collected by Table \ref{error}. Since the approximation is obtained on the 
basis functions $\phi_n$, the discontinuities of the true source terms are 
not exactly recovered. This can be seen from Figure \ref{e2} and the 
comparison between the errors of $(e1)$ and $(e2)$.

\begin{table}[h]
\begin{center}
 \begin{tabular}{|c|c|c|c|c|}
  \hline&$\frac{\|f-f_N\|_{L^2(D)}}{\|f\|_{L^2(D)}}$&$\gamma_f$ &$\frac{\||g|-|g_N|\|_{L^2(D)}}{\|g\|_{L^2(D)}}$&$\gamma_g$\\\hline    
 $(e1)$&6.06e-2&\multirow{2}{*}{Not applicable}&2.46e-2&\multirow{2}{*}{Not applicable}\\\cline{1-2} \cline{4-4}
 $(e2)$&4.88e-1&&5.46e-1&\\\hline
 $(e1a)$&2.55e-1&1e-10&1.06e-1&1e-12\\\hline
 $(e1b)$&2.77e-1&1e-10&7.54e-2&1e-12\\\hline
 $(e1c)$&3.76e-1&1e-10&1.81e-1&1e-11\\\hline
 $(e2a)$&5.20e-1&1e-10&8.35e-1&1e-16\\\hline
 $(e2b)$&6.16e-1&1e-13&6.54e-1&1e-16\\\hline
 $(e2c)$&6.47e-1&1e-13&1.24e-0&1e-16\\\hline
 \end{tabular}
\end{center}
\caption{Relative $L^2$ errors and the regularized parameters.}
\label{error}
\end{table}

\begin{figure}[th!]
	\center
\begin{subfigure}
  \centering
\includegraphics[trim = .5cm .5cm .5cm .8cm, clip=true,height=6.2cm,width=16.2cm]
		{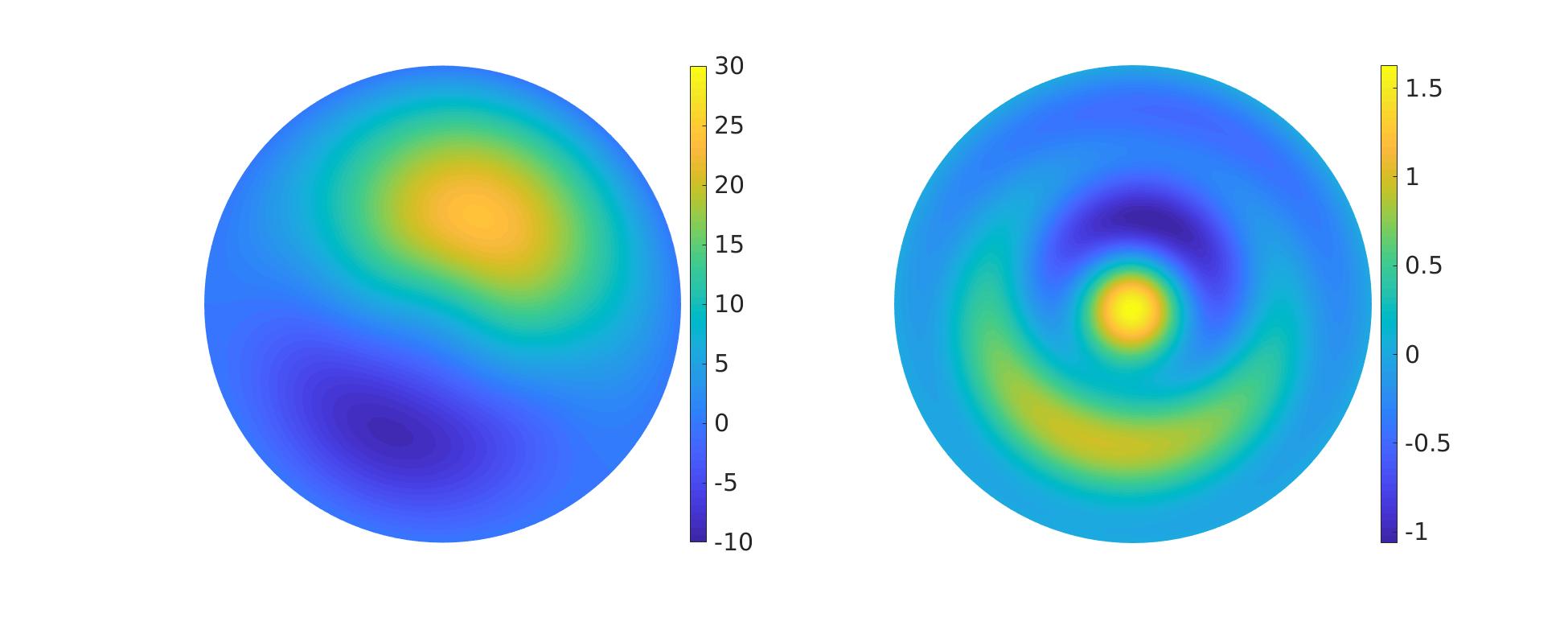}
\end{subfigure}
\\
\begin{subfigure}
  \centering
\includegraphics[trim = .5cm .5cm .5cm .8cm, clip=true,height=6.2cm,width=16.2cm]
		{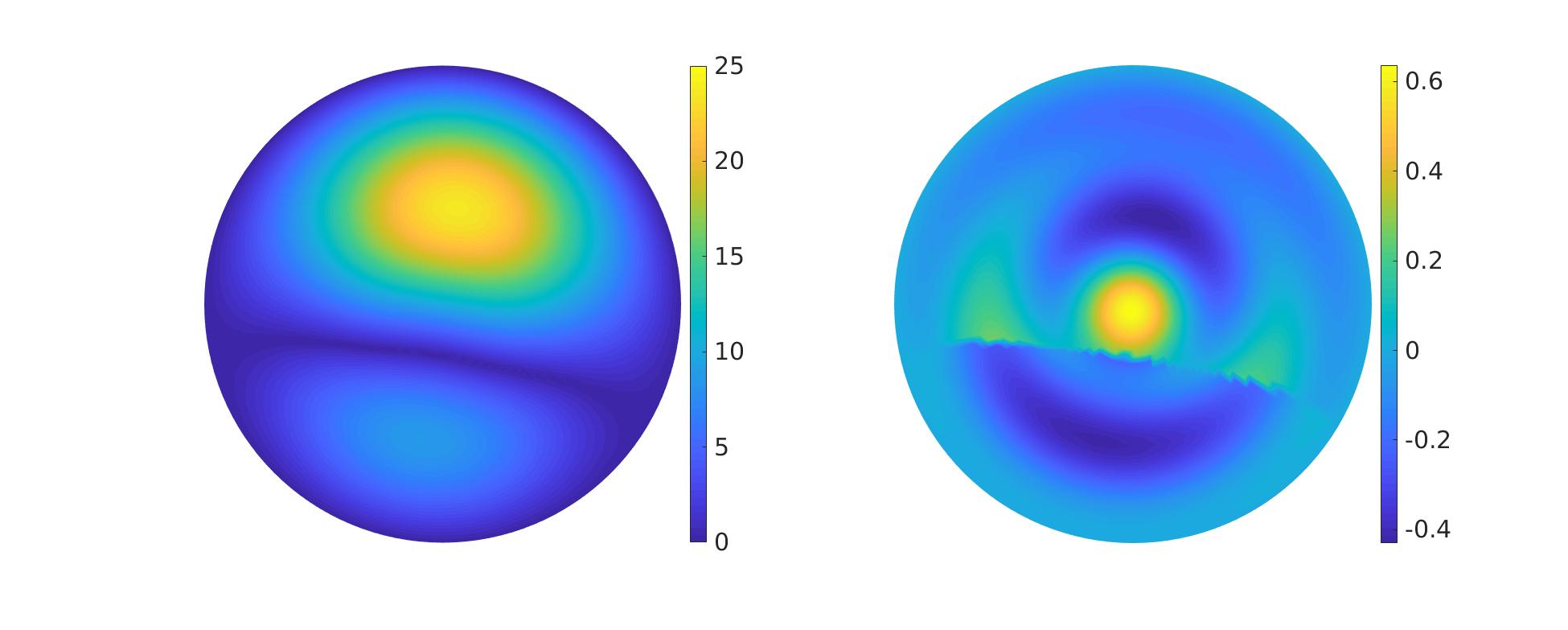}
\end{subfigure}
\caption{\small Experiment $(e1)$ for $f$ (top) and $|g|$ (bottom).\\
Numerical approximation (left), 
difference between exact solution and approximation (right).}
\label{e1}
\end{figure}

\begin{figure}[th!]
	\center
\begin{subfigure}
  \centering
\includegraphics[trim = .5cm .5cm .5cm .8cm, clip=true,height=6.2cm,width=16.2cm]
		{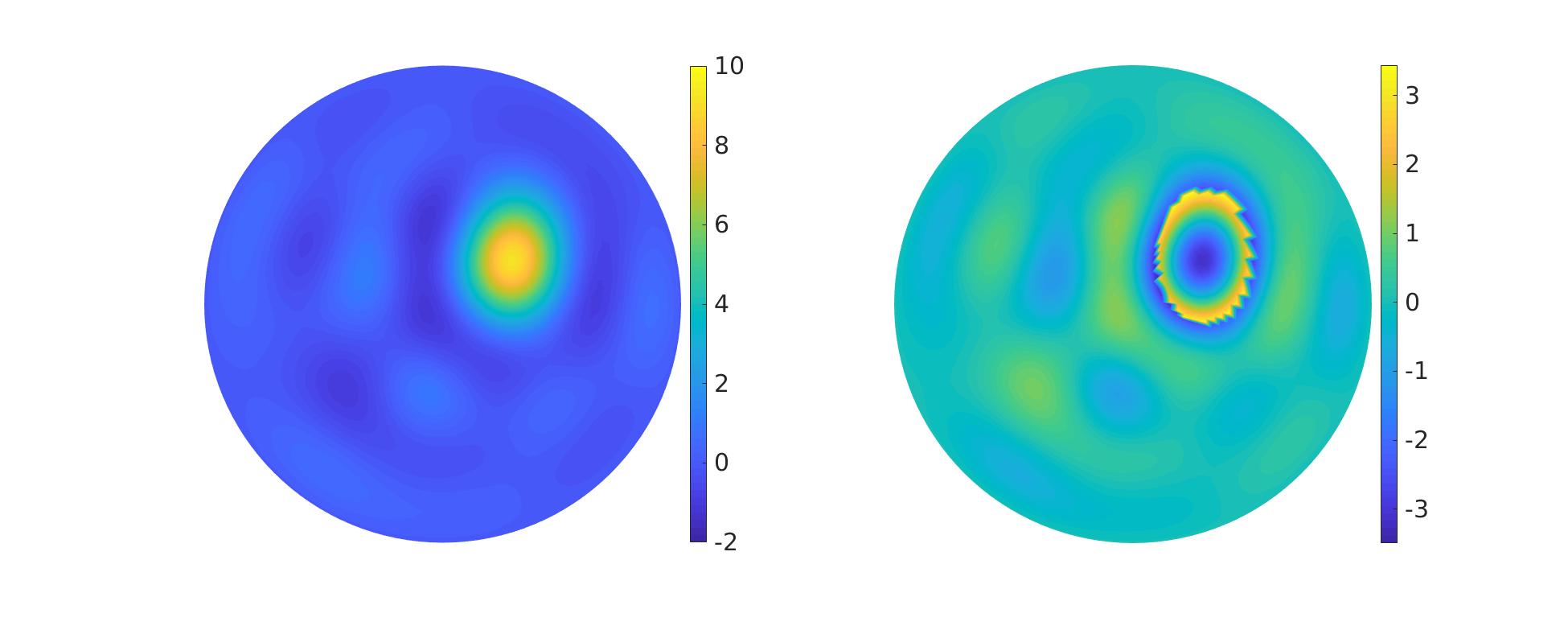}
\end{subfigure}
\\
\begin{subfigure}
  \centering
\includegraphics[trim = .5cm .5cm .5cm .8cm, clip=true,height=6.2cm,width=16.2cm]
		{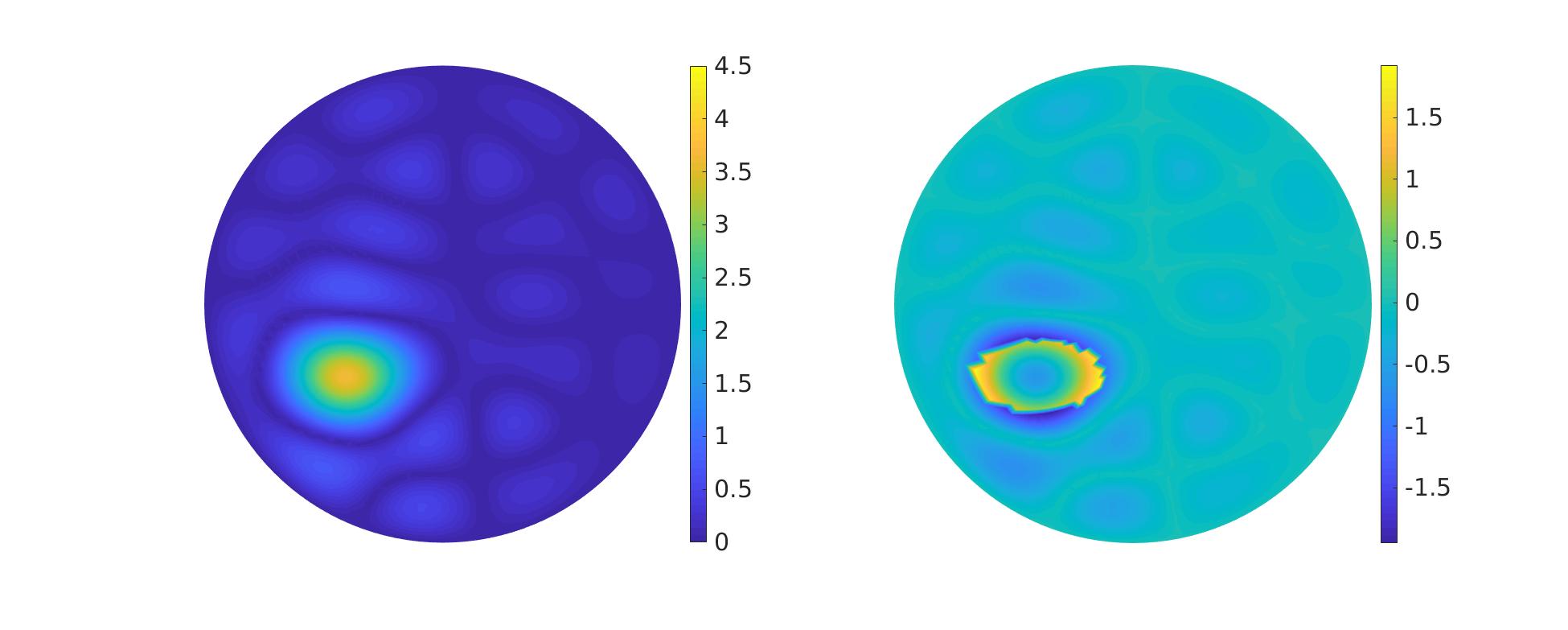}
\end{subfigure}
\caption{\small Experiment $(e2)$ for $f$ (top) and $|g|$ (bottom).\\
Numerical approximation (left), 
difference between exact solution and approximation (right).}
\label{e2}
\end{figure}

\subsection{Numerical results with observations on partial domain}
\label{subsec:partial_domain}
In this subsection, we consider the numerical reconstruction with partial 
measurements, i.e. $D_{mea}\subset D$ and $D_{mea}\ne D.$ 
Here $\{\psi_n\}$ are set as the characteristic functions on each 
uniformly partition of $D_{mea}$ upon the polar coordinates $(r,\theta)$.

Given the noisy data $( \widehat{{\bf E}}^\delta,  \widehat{{\bf C}}^\delta)$, 
for the first equation we set the optimization problem as
\begin{equation*} 
 \argmin_{\vec{f}_N\in \mathbb{R}^N} \Big\{ \|A\vec{f}_N -  
 \widehat{{\bf E}}^\delta\|_{l^2}^2 +\gamma_f\|\vec{f}_N\|_{l^2}^2 \Big\}.
\end{equation*} 
Due to the nonlinearity of the second equation, we choose the 
Levenberg-Marquardt type Newton's iteration
\begin{equation*}
 \vec{g}_{l+1}=\vec{g}_l+[B'(\vec{g}_l)^\top B'(\vec{g}_l)+\gamma_gI_N]^{-1}
  B'(\vec{g}_l)^\top (\widehat{{\bf C}}^\delta-B\vec{g}_l),
\end{equation*}
and the Frechet derivative $B'$ of $B$ is given as 
\begin{equation*}
 B'(g)[h]=\int_0^T {\bf R}^\top [{\bf g}(\tau) {\bf h}(\tau)^\top
 +{\bf h}(\tau) {\bf g}(\tau)^\top] {\bf R}\ d\tau.
\end{equation*}

\begin{figure}[h]
	\center
\begin{subfigure}
  \centering
\includegraphics[trim = .5cm .5cm .5cm .8cm, clip=true,height=6.2cm,width=16.2cm]
		{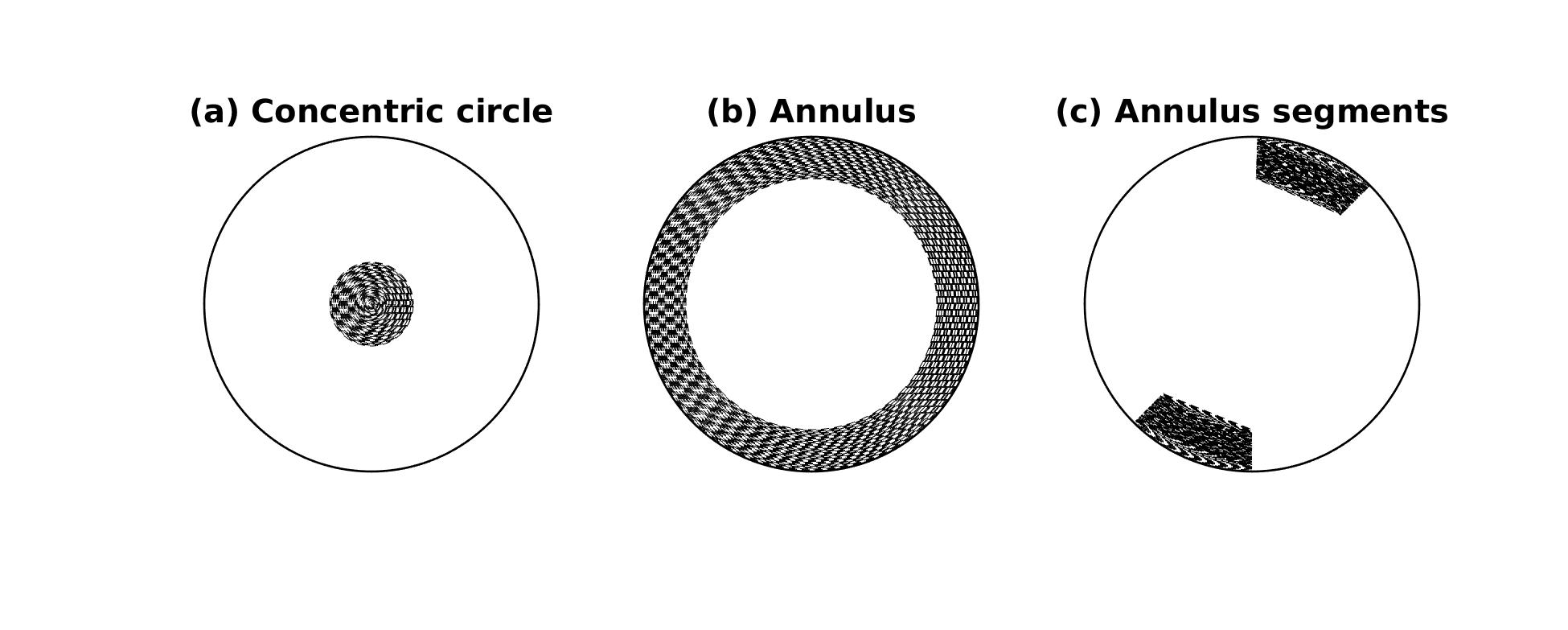}
\end{subfigure}
\caption{\small Three partial domains $D_{mea}$ in section 
\ref{subsec:partial_domain}. In each case the shaded area is observed.}
\label{subdomain}
\end{figure}
We try three kinds of subsets of $D$ which are set as the observed area and 
can be seen in Figure \ref{subdomain}. $(a)$ is a concentric with radius $1/4$, 
$(b)$ is the annulus between the circles with radius $3/4$ and $1$, and 
$(c)$ contains two segments of the annulus in $(b)$ with $\pi/4$ radian span.
The exact solutions and the parameter setting \eqref{parameter} 
in experiments $(e1)$ and $(e2)$ are still used but the corresponding notations 
are changed to $(e1a),\ (e1b),\ (e1c)$ 
and $(e2a),\ (e2b),\ (e2c).$ The results are displayed in Figures  
\ref{e1_f_partial}, \ref{e1_g_partial}, \ref{e2_f_partial}, \ref{e2_g_partial}, 
and the relative $L^2$ errors are recorded in Table \ref{error}. 
In these experiments, the values of the regularized parameters 
$\gamma_f, \gamma_g$ are chosen empirically.

Similar to the results of experiments $(e1)$ and $(e2)$, 
the reconstructions for smooth exact 
solutions are better than the ones for nonsmooth case. Furthermore, due to the 
lack of measured data, the performance of experiments 
$\{(e1j),(e2j): j=a,b,c\}$ is worse than $(e1)$ and $(e2)$, and this can be seen 
in Figures \ref{e1_f_partial}, \ref{e1_g_partial}, \ref{e2_f_partial}, 
\ref{e2_g_partial} and Table \ref{error}. Also, the results for $(e2c)$ 
show that the observed subdomain $(c)$ in Figure \ref{subdomain} for the 
discontinuous case is close to the limit in terms of noise level and the size of the subdomain of which can ensure a useful 
localization of the source.

\begin{figure}[th!]
	\center
\begin{subfigure}
\centering
\includegraphics[trim = .5cm .5cm .5cm .8cm, clip=true,height=6.2cm,width=16.2cm]
		{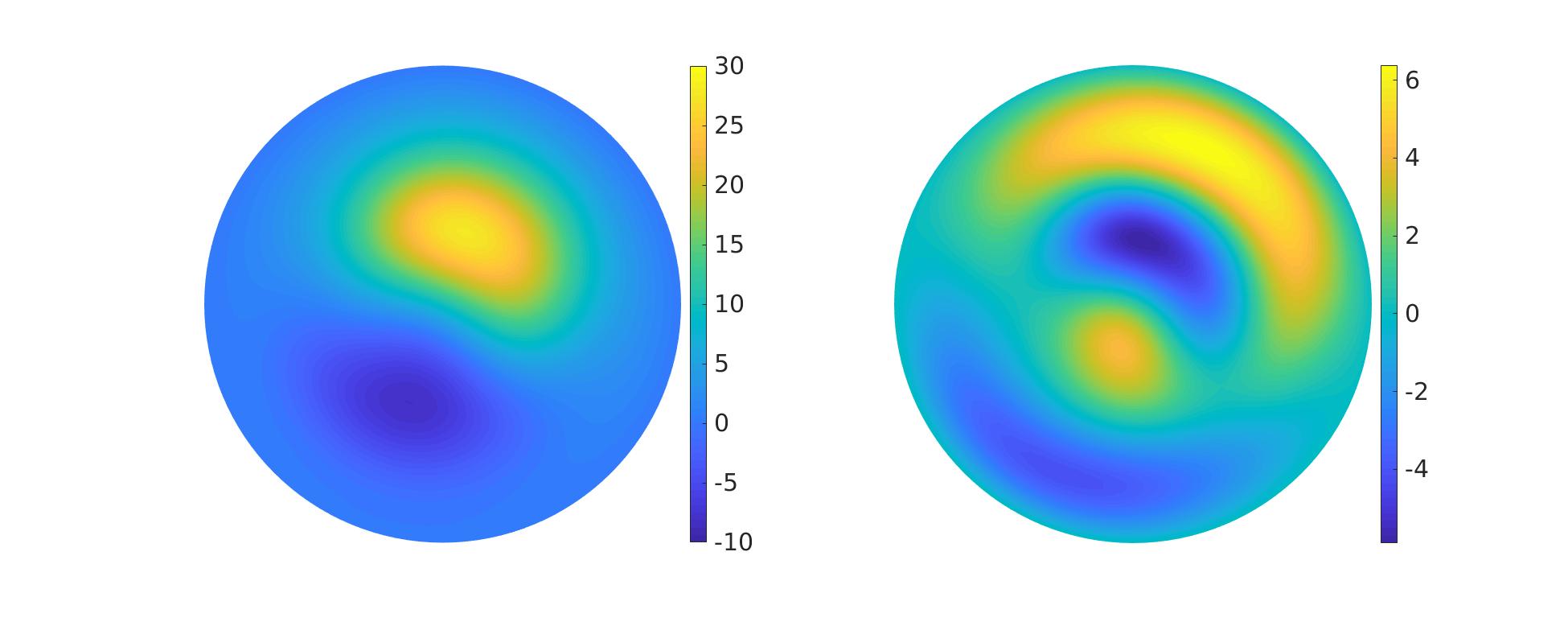}
\end{subfigure}
\\
\begin{subfigure}
  \centering
\includegraphics[trim = .5cm .5cm .5cm .8cm, clip=true,height=6.2cm,width=16.2cm]
		{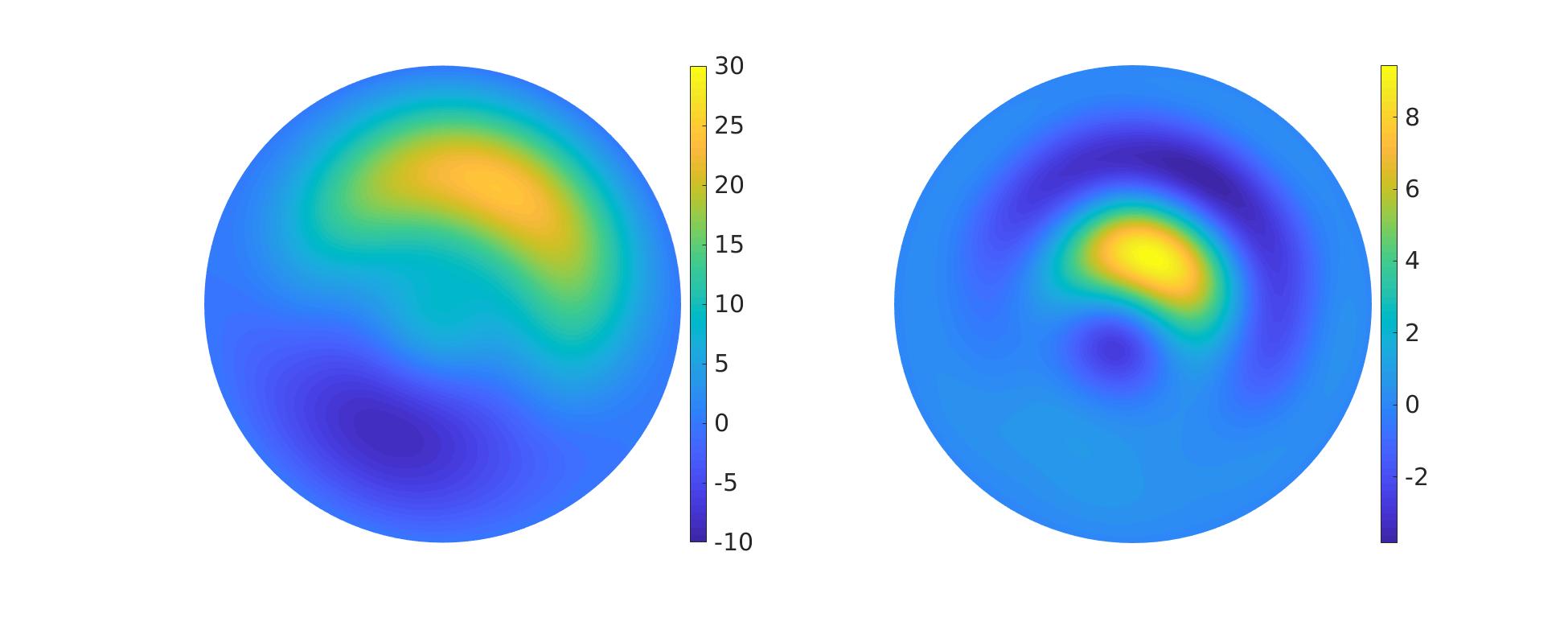}
\end{subfigure}
\\
\begin{subfigure}
  \centering
\includegraphics[trim = .5cm .5cm .5cm .8cm, clip=true,height=6.2cm,width=16.2cm]
		{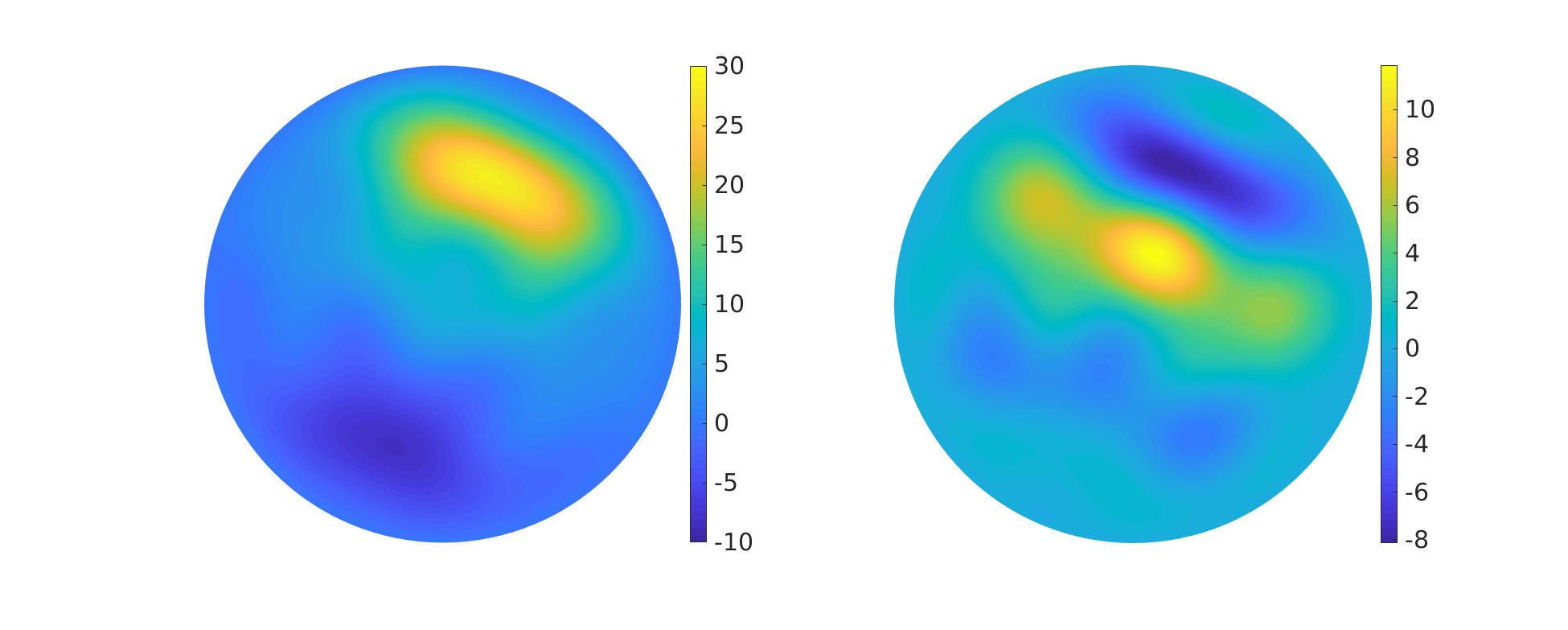}
\end{subfigure}
\caption{\small Reconstruction for $f$ with experiment $(e1a)$ (top), 
$(e1b)$ (middle) and $(e1c)$ (bottom).
\\ Numerical approximation (left), 
difference between exact solution and approximation (right).}
\label{e1_f_partial}
\end{figure}

\begin{figure}[th!]
	\center
\begin{subfigure}
  \centering
\includegraphics[trim = .5cm .5cm .5cm .8cm, clip=true,height=6.2cm,width=16.2cm]
		{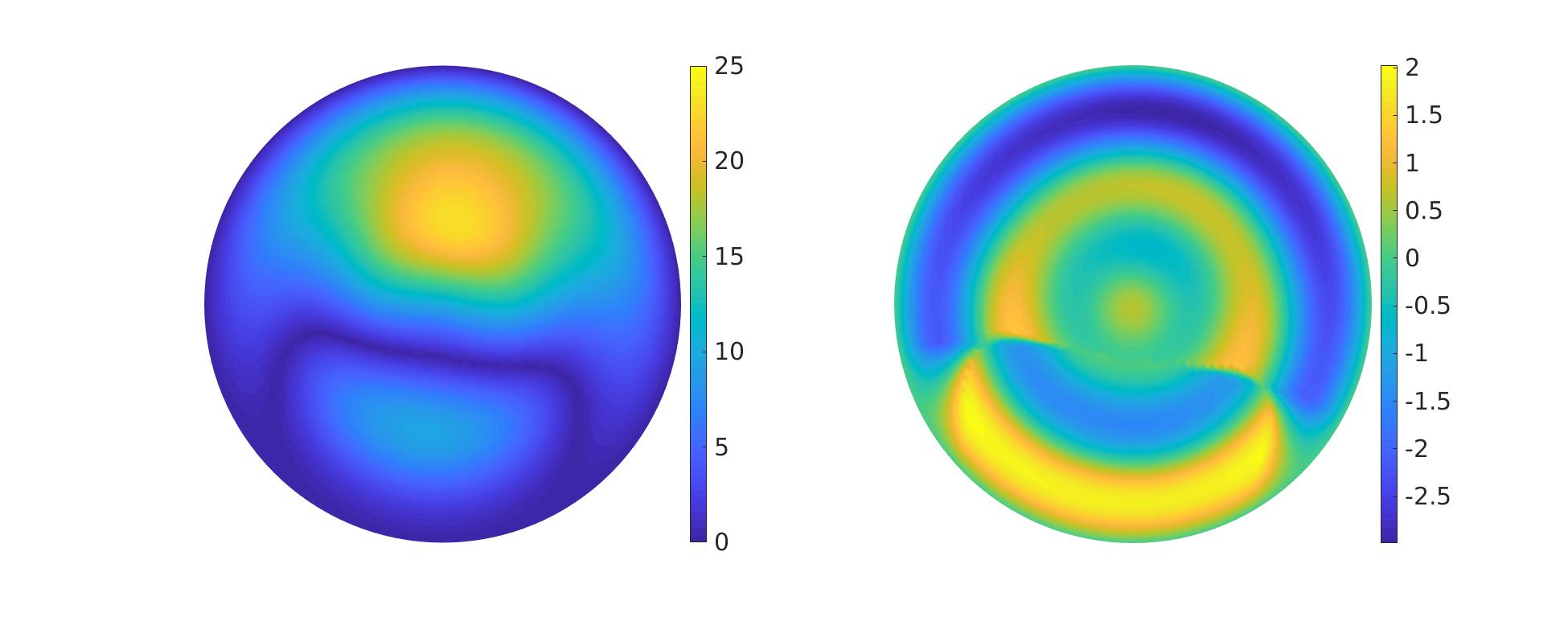}
\end{subfigure}
\\
\begin{subfigure}
  \centering
\includegraphics[trim = .5cm .5cm .5cm .8cm, clip=true,height=6.2cm,width=16.2cm]
		{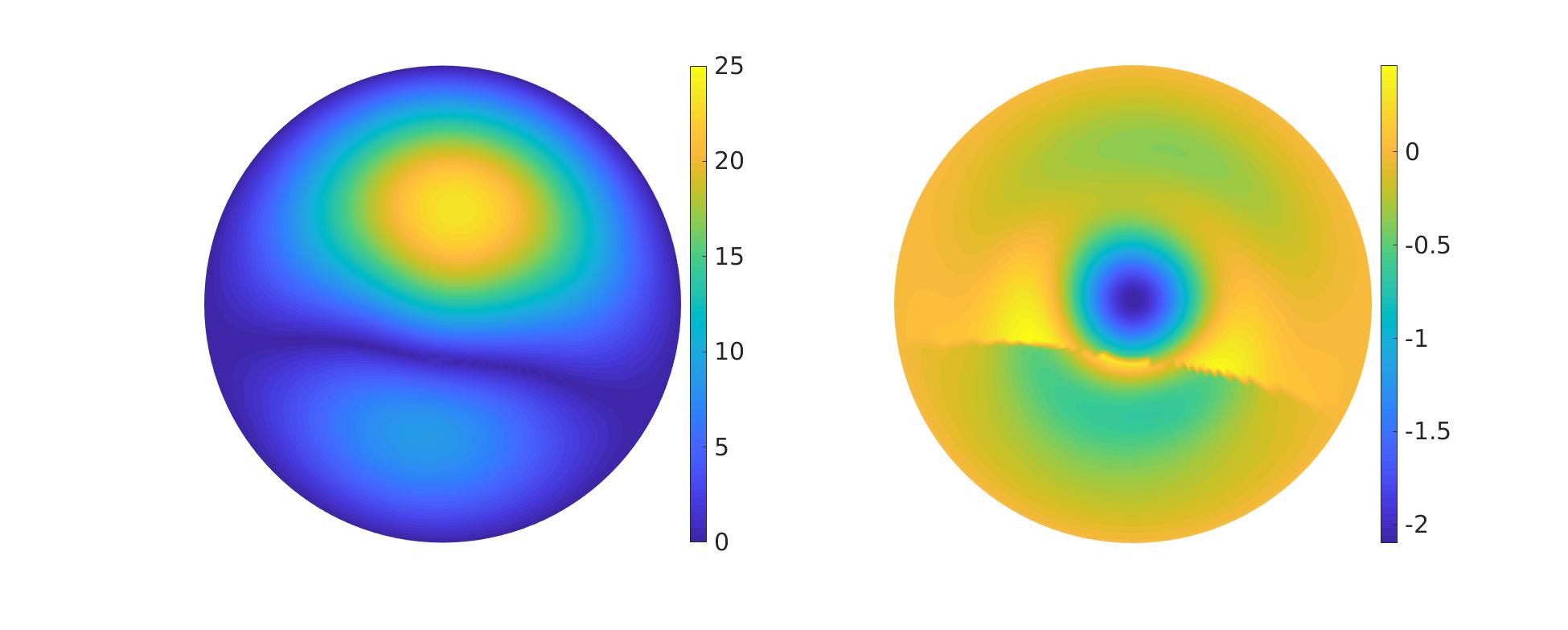}
\end{subfigure}
\\
\begin{subfigure}
  \centering
\includegraphics[trim = .5cm .5cm .5cm .8cm, clip=true,height=6.2cm,width=16.2cm]
		{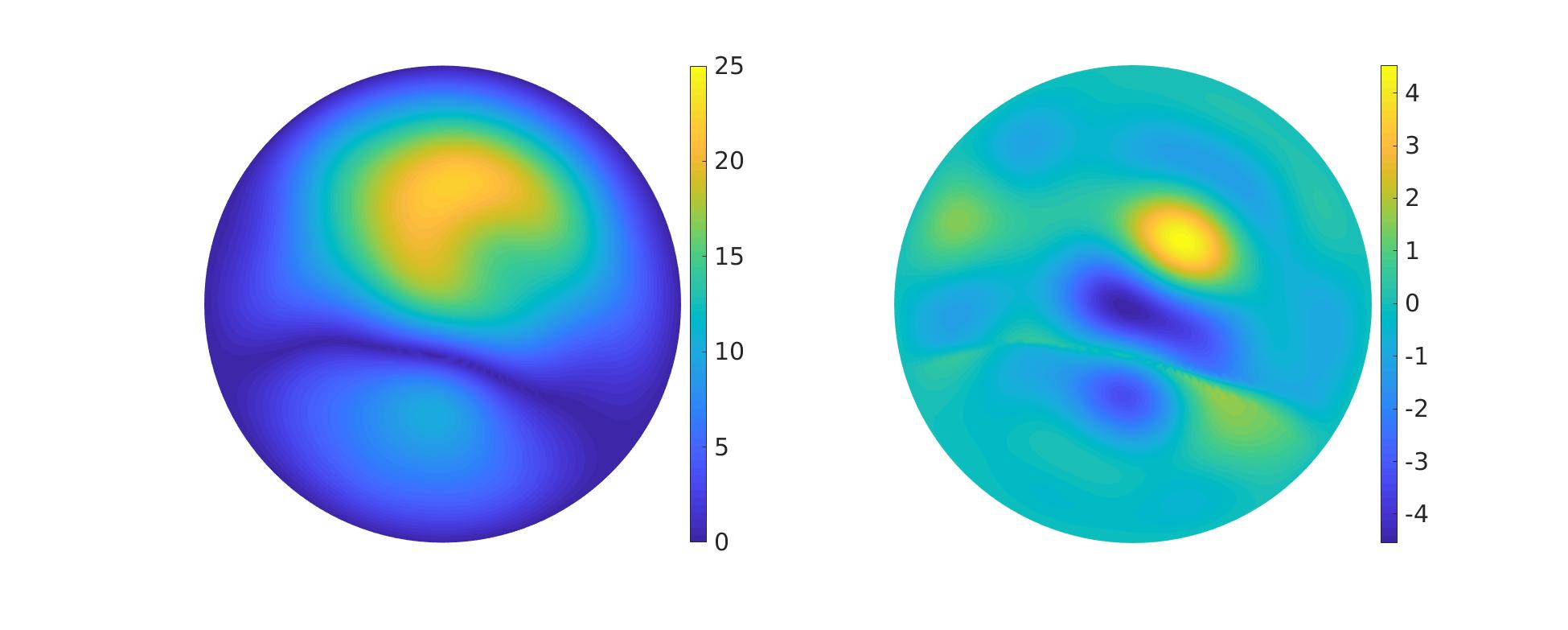}
\end{subfigure}
\caption{\small Reconstruction for $|g|$ with experiment $(e1a)$ (top), 
$(e1b)$ (middle) and $(e1c)$ (bottom).
\\ Numerical approximation (left), 
difference between exact solution and approximation (right).}
\label{e1_g_partial}
\end{figure}

\begin{figure}[th!]
	\center
\begin{subfigure}
  \centering
\includegraphics[trim = .5cm .5cm .5cm .8cm, clip=true,height=6.2cm,width=16.2cm]
		{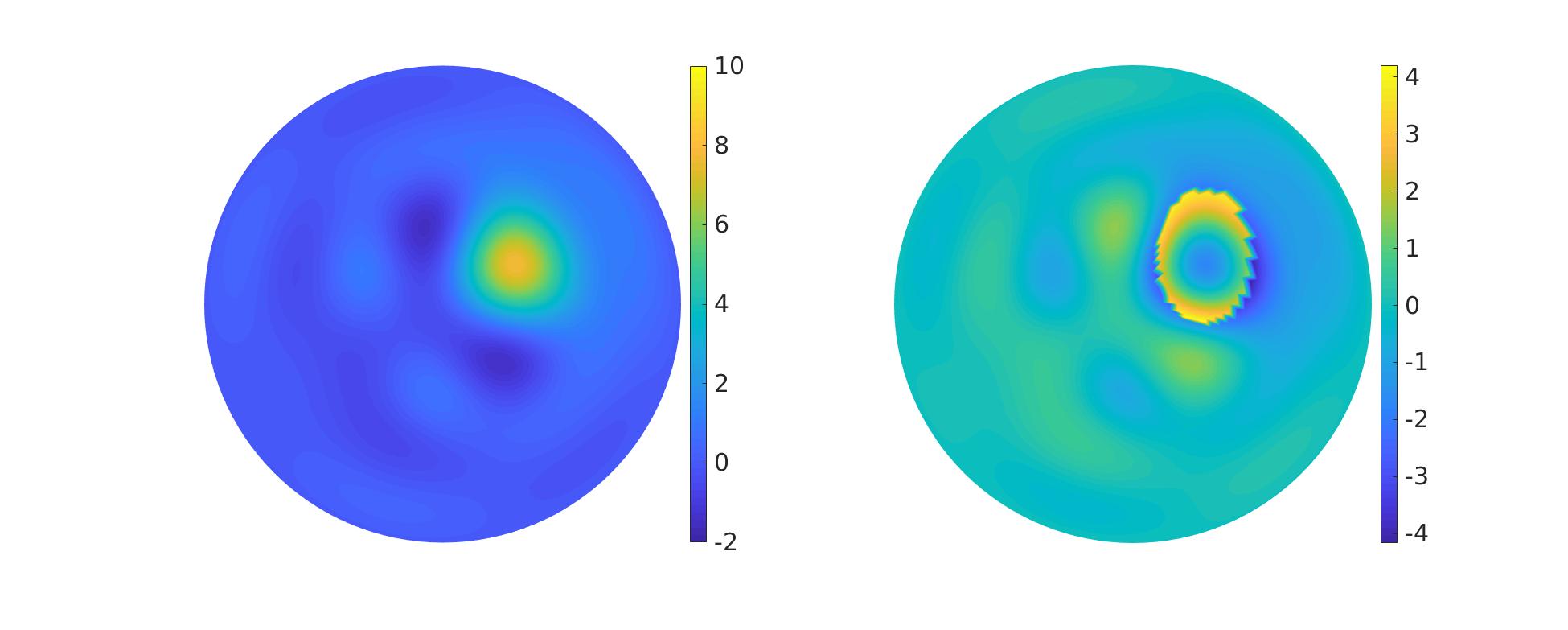}
\end{subfigure}
\\
\begin{subfigure}
  \centering
\includegraphics[trim = .5cm .5cm .5cm .8cm, clip=true,height=6.2cm,width=16.2cm]
		{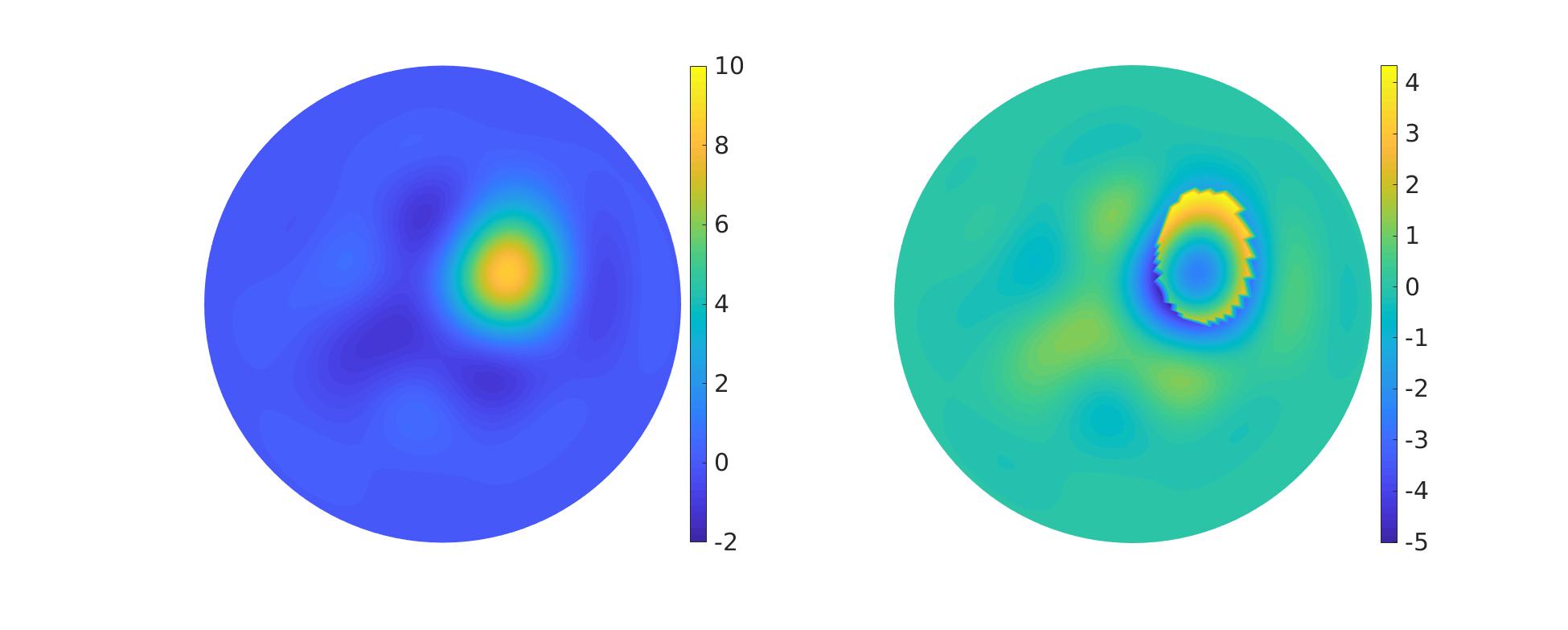}
\end{subfigure}
\\
\begin{subfigure}
  \centering
\includegraphics[trim = .5cm .5cm .5cm .8cm, clip=true,height=6.2cm,width=16.2cm]
		{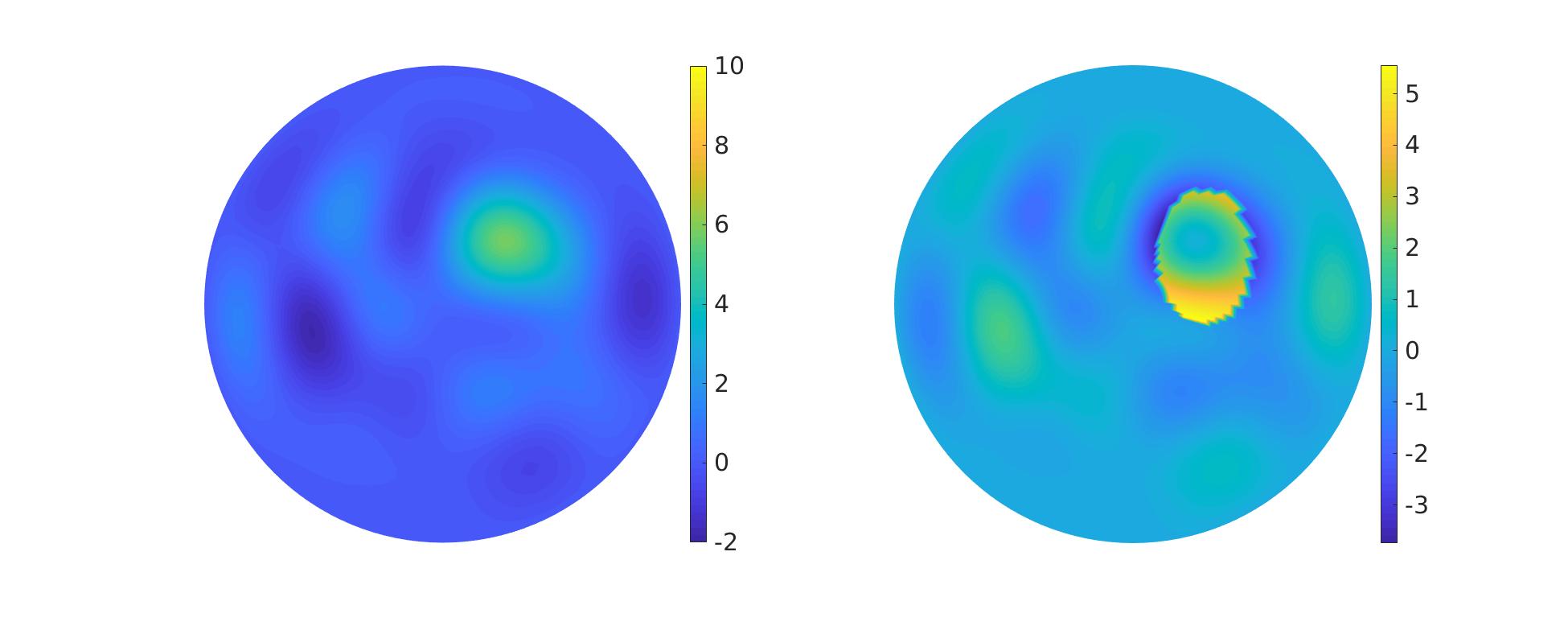}
\end{subfigure}
\caption{\small Reconstruction for $f$ with experiment $(e2a)$ (top), 
$(e2b)$ (middle) and $(e2c)$ (bottom).
\\ Numerical approximation (left), 
difference between exact solution and approximation (right).}
\label{e2_f_partial}
\end{figure}

\begin{figure}[th!]
	\center
\begin{subfigure}
  \centering
\includegraphics[trim = .5cm .5cm .5cm .8cm, clip=true,height=6.2cm,width=16.2cm]
		{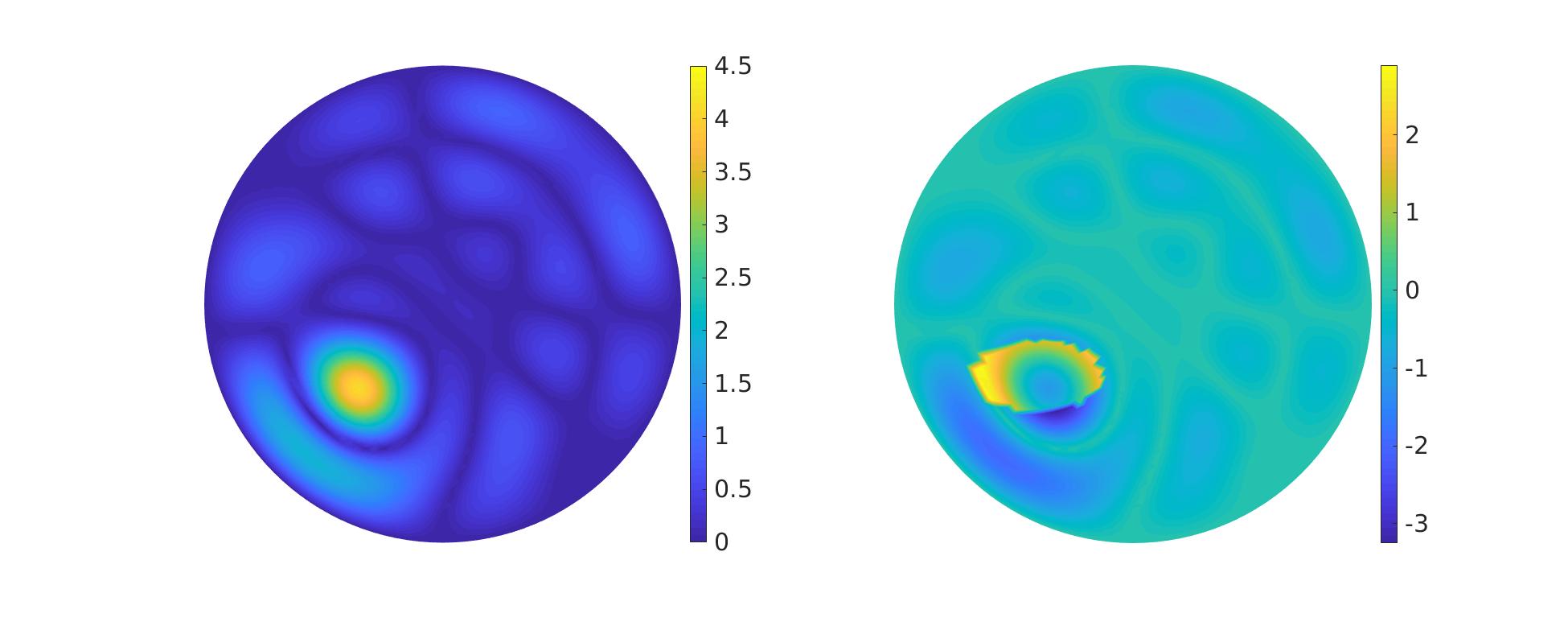}
\end{subfigure}
\\
\begin{subfigure}
  \centering
\includegraphics[trim = .5cm .5cm .5cm .8cm, clip=true,height=6.2cm,width=16.2cm]
		{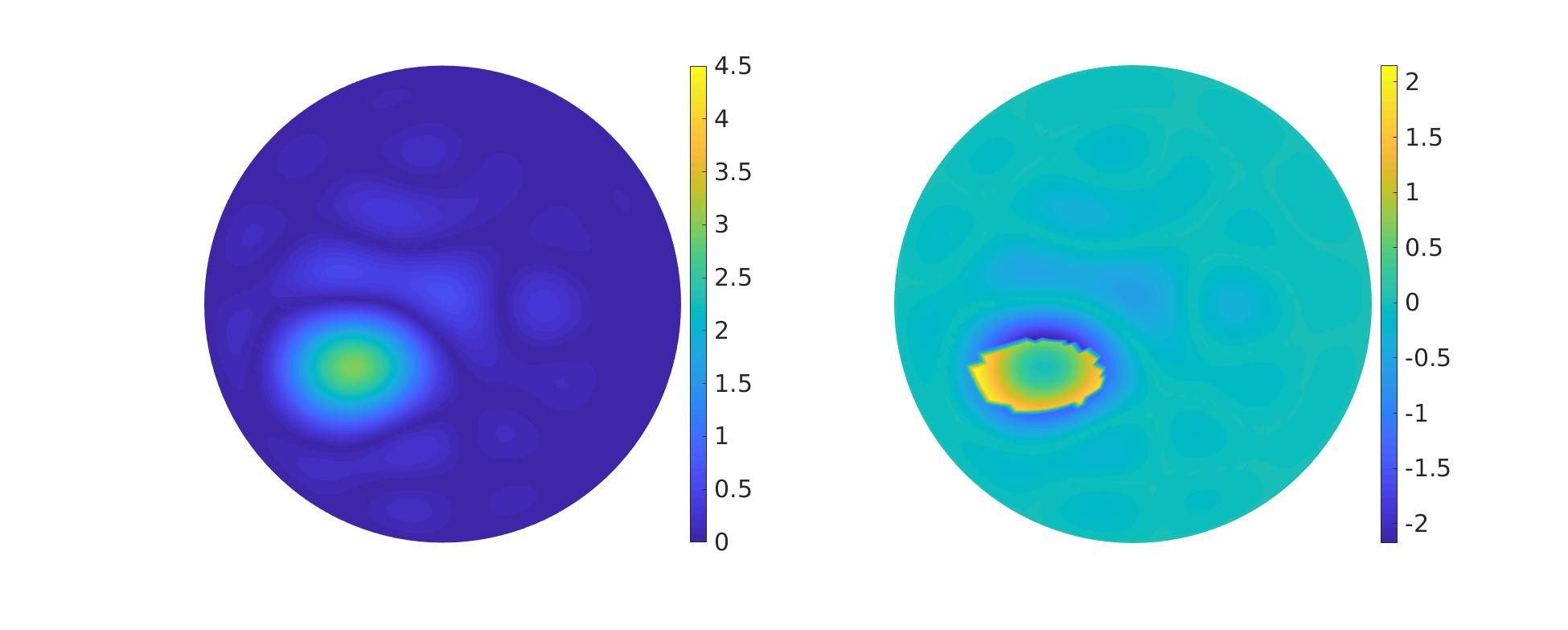}
\end{subfigure}
\\
\begin{subfigure}
  \centering
\includegraphics[trim = .5cm .5cm .5cm .8cm, clip=true,height=6.2cm,width=16.2cm]
		{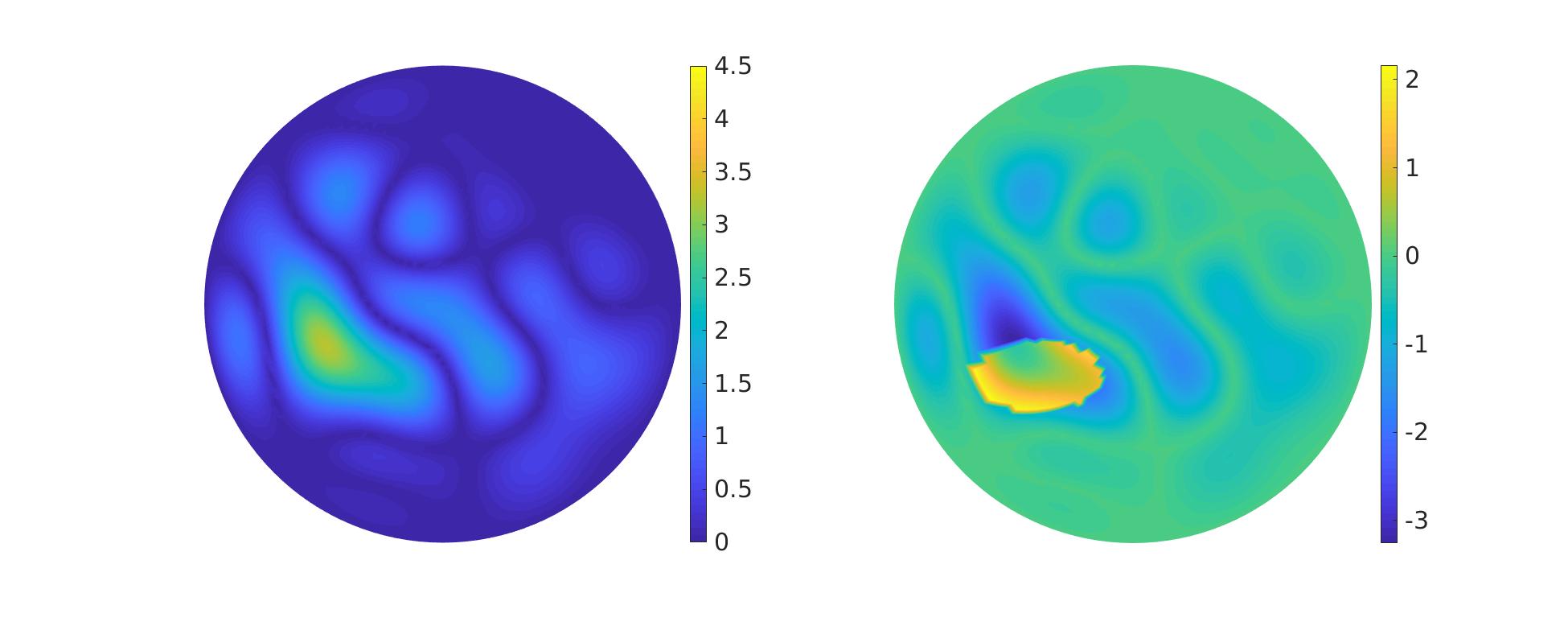}
\end{subfigure}
\caption{\small Reconstruction for $|g|$ with experiment $(e2a)$ (top), 
$(e2b)$ (middle) and $(e2c)$ (bottom).
\\ Numerical approximation (left), 
difference between exact solution and approximation (right).}
\label{e2_g_partial}
\end{figure}

\section*{Acknowledgement}

PN is partially supported by China Scholarship Council and Finnish National Agency for 
Education (ID:201702720003), NSFC (key projects no.11331004, no.11421110002) 
and the Programme of Introducing Talents of Discipline to Universities (no.B08018).
TH and ZZ were supported by the Three-year grant "Stochastic inverse problems 
in atmospheric tomography" of the University of Helsinki.
In addition, TH was supported by the Academy of Finland via projects 275177 and 314879.

\bibliographystyle{abbrvurl} 
\bibliography{SFDE_inversesource}

\begin{thebibliography}{10}

\bibitem{adams1992field}
E.~E. Adams and L.~W. Gelhar.
\newblock Field study of dispersion in a heterogeneous aquifer: 2. spatial
  moments analysis.
\newblock {\em Water Resources Research}, 28(12):3293--3307, 1992.

\bibitem{Bao2016several}
G.~Bao, C.~Chen, and P.~Li.
\newblock Inverse random source scattering problems in several dimensions.
\newblock {\em SIAM/ASA J. Uncertain. Quantif.}, 4(1):1263--1287, 2016.
\newblock URL: \url{https://doi.org/10.1137/16M1067470}.

\bibitem{bao2017inverse}
G.~Bao, C.~Chen, and P.~Li.
\newblock Inverse random source scattering for elastic waves.
\newblock {\em SIAM Journal on Numerical Analysis}, 55(6):2616--2643, 2017.

\bibitem{Bao2010numerical}
G.~Bao, S.-N. Chow, P.~Li, and H.~Zhou.
\newblock Numerical solution of an inverse medium scattering problem with a
  stochastic source.
\newblock {\em Inverse Problems}, 26(7):074014, 2010.

\bibitem{Bao2014Helmholtz}
G.~Bao, S.-N. Chow, P.~Li, and H.~Zhou.
\newblock An inverse random source problem for the {H}elmholtz equation.
\newblock {\em Math. Comp.}, 83(285):215--233, 2014.
\newblock URL: \url{https://doi.org/10.1090/S0025-5718-2013-02730-5}.

\bibitem{BarkaiMetzlerKlafter:2000}
E.~Barkai, R.~Metzler, and J.~Klafter.
\newblock From continuous time random walks to the fractional fokker-planck
  equation.
\newblock {\em Phys. Rev. E}, 61:132--138, Jan 2000.
\newblock URL: \url{https://link.aps.org/doi/10.1103/PhysRevE.61.132}, \href
  {http://dx.doi.org/10.1103/PhysRevE.61.132}
  {\path{doi:10.1103/PhysRevE.61.132}}.

\bibitem{berkowitz2006modeling}
B.~Berkowitz, A.~Cortis, M.~Dentz, and H.~Scher.
\newblock Modeling non-fickian transport in geological formations as a
  continuous time random walk.
\newblock {\em Reviews of Geophysics}, 44(2), 2006.

\bibitem{borcea2002imaging}
L.~Borcea, G.~Papanicolaou, C.~Tsogka, and J.~Berryman.
\newblock Imaging and time reversal in random media.
\newblock {\em Inverse Problems}, 18(5):1247, 2002.

\bibitem{BouchaudGeorges:1990}
J.-P. Bouchaud and A.~Georges.
\newblock Anomalous diffusion in disordered media: Statistical mechanisms,
  models and physical applications.
\newblock {\em Physics Reports}, 195(4):127 -- 293, 1990.
\newblock URL:
  \url{http://www.sciencedirect.com/science/article/pii/037015739090099N},
  \href {http://dx.doi.org/https://doi.org/10.1016/0370-1573(90)90099-N}
  {\path{doi:https://doi.org/10.1016/0370-1573(90)90099-N}}.

\bibitem{caro2016inverse}
P.~Caro, T.~Helin, and M.~Lassas.
\newblock Inverse scattering for a random potential.
\newblock {\em arXiv preprint arXiv:1605.08710}, 2016.

\bibitem{cheng2009uniqueness}
J.~Cheng, J.~Nakagawa, M.~Yamamoto, and T.~Yamazaki.
\newblock Uniqueness in an inverse problem for a one-dimensional fractional
  diffusion equation.
\newblock {\em Inverse Problems}, 25(11):115002, 16, 2009.
\newblock URL: \url{http://dx.doi.org/10.1088/0266-5611/25/11/115002}, \href
  {http://dx.doi.org/10.1088/0266-5611/25/11/115002}
  {\path{doi:10.1088/0266-5611/25/11/115002}}.

\bibitem{devaney1979inverse}
A.~Devaney.
\newblock The inverse problem for random sources.
\newblock {\em Journal of Mathematical Physics}, 20(8):1687--1691, 1979.

\bibitem{einstein1905molekularkinetischen}
A.~Einstein.
\newblock {\"U}ber die von der molekularkinetischen theorie der w{\"a}rme
  geforderte bewegung von in ruhenden fl{\"u}ssigkeiten suspendierten teilchen.
\newblock {\em Annalen der physik}, 322(8):549--560, 1905.

\bibitem{el2002some}
M.~M. El-Borai.
\newblock Some probability densities and fundamental solutions of fractional
  evolution equations.
\newblock {\em Chaos, Solitons \& Fractals}, 14(3):433--440, 2002.

\bibitem{garnier2009passive}
J.~Garnier and G.~Papanicolaou.
\newblock Passive sensor imaging using cross correlations of noisy signals in a
  scattering medium.
\newblock {\em SIAM Journal on Imaging Sciences}, 2(2):396--437, 2009.

\bibitem{garnier2012correlation}
J.~Garnier and G.~Papanicolaou.
\newblock Correlation-based virtual source imaging in strongly scattering
  random media.
\newblock {\em Inverse Problems}, 28(7):075002, 2012.

\bibitem{garnier2016passive}
J.~Garnier and G.~Papanicolaou.
\newblock {\em Passive imaging with ambient noise}.
\newblock Cambridge University Press, 2016.

\bibitem{GefenAharonyAlexander:1983}
Y.~Gefen, A.~Aharony, and S.~Alexander.
\newblock Anomalous diffusion on percolating clusters.
\newblock {\em Phys. Rev. Lett.}, 50:77--80, Jan 1983.
\newblock URL: \url{https://link.aps.org/doi/10.1103/PhysRevLett.50.77}, \href
  {http://dx.doi.org/10.1103/PhysRevLett.50.77}
  {\path{doi:10.1103/PhysRevLett.50.77}}.

\bibitem{gorenflo2014mittag}
R.~Gorenflo, A.~A. Kilbas, F.~Mainardi, and S.~V. Rogosin.
\newblock {\em Mittag-{L}effler functions, related topics and applications}.
\newblock Springer Monographs in Mathematics. Springer, Heidelberg, 2014.
\newblock URL: \url{http://dx.doi.org/10.1007/978-3-662-43930-2}, \href
  {http://dx.doi.org/10.1007/978-3-662-43930-2}
  {\path{doi:10.1007/978-3-662-43930-2}}.

\bibitem{hatano1998dispersive}
Y.~Hatano and N.~Hatano.
\newblock Dispersive transport of ions in column experiments: An explanation of
  long-tailed profiles.
\newblock {\em Water resources research}, 34(5):1027--1033, 1998.

\bibitem{helin2018atmospheric}
T.~Helin, S.~Kindermann, J.~Lehtonen, and R.~Ramlau.
\newblock Atmospheric turbulence profiling with unknown power spectral density.
\newblock {\em Inverse Problems}, 2018.

\bibitem{helin2016correlation}
T.~Helin, M.~Lassas, L.~Oksanen, and T.~Saksala.
\newblock Correlation based passive imaging with a white noise source.
\newblock {\em arXiv preprint arXiv:1609.08022}, 2016.

\bibitem{helin2017inverse}
T.~Helin, M.~Lassas, and L.~P{\"a}iv{\"a}rinta.
\newblock Inverse acoustic scattering problem in half-space with anisotropic
  random impedance.
\newblock {\em Journal of Differential Equations}, 262(4):3139--3168, 2017.

\bibitem{JinLazarovZhou:2016}
B.~Jin, R.~Lazarov, and Z.~Zhou.
\newblock An analysis of the {L}1 scheme for the subdiffusion equation with
  nonsmooth data.
\newblock {\em IMA J. Numer. Anal.}, 36(1):197--221, 2016.
\newblock URL: \url{http://dx.doi.org/10.1093/imanum/dru063}, \href
  {http://dx.doi.org/10.1093/imanum/dru063} {\path{doi:10.1093/imanum/dru063}}.

\bibitem{Jin2015tutorial}
B.~Jin and W.~Rundell.
\newblock A tutorial on inverse problems for anomalous diffusion processes.
\newblock {\em Inverse Problems}, 31(3):035003, 40, 2015.
\newblock URL: \url{https://doi.org/10.1088/0266-5611/31/3/035003}.

\bibitem{kilbas2006theory}
A.~A. Kilbas, H.~M. Srivastava, and J.~J. Trujillo.
\newblock {\em Theory and applications of fractional differential equations},
  volume 204 of {\em North-Holland Mathematics Studies}.
\newblock Elsevier Science B.V., Amsterdam, 2006.

\bibitem{KlafterSilbey:1980}
J.~Klafter and R.~Silbey.
\newblock Derivation of the continuous-time random-walk equation.
\newblock {\em Phys. Rev. Lett.}, 44:55--58, Jan 1980.
\newblock URL: \url{https://link.aps.org/doi/10.1103/PhysRevLett.44.55}, \href
  {http://dx.doi.org/10.1103/PhysRevLett.44.55}
  {\path{doi:10.1103/PhysRevLett.44.55}}.

\bibitem{li2017inverse}
M.~Li, C.~Chen, and P.~Li.
\newblock Inverse random source scattering for the helmholtz equation in
  inhomogeneous media.
\newblock {\em Inverse Problems}, 34(1):015003, 2017.

\bibitem{Li2011source}
P.~Li.
\newblock An inverse random source scattering problem in inhomogeneous media.
\newblock {\em Inverse Problems}, 27(3):035004, 22, 2011.
\newblock URL: \url{https://doi.org/10.1088/0266-5611/27/3/035004}.

\bibitem{li2017stability}
P.~Li and G.~Yuan.
\newblock Stability on the inverse random source scattering problem for the
  one-dimensional helmholtz equation.
\newblock {\em Journal of Mathematical Analysis and Applications},
  450(2):872--887, 2017.

\bibitem{Li2017analyticity}
Z.~Li, Y.~Luchko, and M.~Yamamoto.
\newblock Analyticity of solutions to a distributed order time-fractional
  diffusion equation and its application to an inverse problem.
\newblock {\em Comput. Math. Appl.}, 73(6):1041--1052, 2017.
\newblock URL: \url{https://doi.org/10.1016/j.camwa.2016.06.030}.

\bibitem{Liu2017reconstruction}
Y.~Liu and Z.~Zhang.
\newblock Reconstruction of the temporal component in the source term of a
  (time-fractional) diffusion equation.
\newblock {\em Journal of Physics A: Mathematical and Theoretical},
  50(30):305203, 2017.
\newblock URL: \url{http://stacks.iop.org/1751-8121/50/i=30/a=305203}.

\bibitem{Luchko2009maximum}
Y.~Luchko.
\newblock Maximum principle for the generalized time-fractional diffusion
  equation.
\newblock {\em J. Math. Anal. Appl.}, 351(1):218--223, 2009.
\newblock URL: \url{https://doi.org/10.1016/j.jmaa.2008.10.018}.

\bibitem{Luchko2011maximum}
Y.~Luchko.
\newblock Maximum principle and its application for the time-fractional
  diffusion equations.
\newblock {\em Fract. Calc. Appl. Anal.}, 14(1):110--124, 2011.
\newblock URL: \url{https://doi.org/10.2478/s13540-011-0008-6}.

\bibitem{mainardi2010fractional}
F.~Mainardi.
\newblock {\em Fractional calculus and waves in linear viscoelasticity}.
\newblock Imperial College Press, London, 2010.
\newblock An introduction to mathematical models.
\newblock URL: \url{http://dx.doi.org/10.1142/9781848163300}, \href
  {http://dx.doi.org/10.1142/9781848163300} {\path{doi:10.1142/9781848163300}}.

\bibitem{MetzlerKlafter:2000}
R.~Metzler and J.~Klafter.
\newblock The random walk's guide to anomalous diffusion: a fractional dynamics
  approach.
\newblock {\em Physics Reports}, 339(1):1 -- 77, 2000.
\newblock URL:
  \url{http://www.sciencedirect.com/science/article/pii/S0370157300000703},
  \href {http://dx.doi.org/https://doi.org/10.1016/S0370-1573(00)00070-3}
  {\path{doi:https://doi.org/10.1016/S0370-1573(00)00070-3}}.

\bibitem{mijena2015space}
J.~B. Mijena and E.~Nane.
\newblock Space--time fractional stochastic partial differential equations.
\newblock {\em Stochastic Processes and their Applications}, 125(9):3301--3326,
  2015.

\bibitem{nigmatullin1986realization}
R.~Nigmatullin.
\newblock The realization of the generalized transfer equation in a medium with
  fractal geometry.
\newblock {\em physica status solidi (b)}, 133(1):425--430, 1986.

\bibitem{Bernt2003stochastic}
B.~\O~ksendal.
\newblock {\em Stochastic differential equations}.
\newblock Universitext. Springer-Verlag, Berlin, sixth edition, 2003.
\newblock An introduction with applications.
\newblock URL: \url{https://doi.org/10.1007/978-3-642-14394-6}.

\bibitem{Podlubny1999fractional}
I.~Podlubny.
\newblock {\em Fractional differential equations}, volume 198 of {\em
  Mathematics in Science and Engineering}.
\newblock Academic Press, Inc., San Diego, CA, 1999.
\newblock An introduction to fractional derivatives, fractional differential
  equations, to methods of their solution and some of their applications.

\bibitem{pollard1948completely}
H.~Pollard.
\newblock The completely monotonic character of the {M}ittag-{L}effler function
  {$E_a(-x)$}.
\newblock {\em Bull. Amer. Math. Soc.}, 54:1115--1116, 1948.
\newblock URL: \url{http://dx.doi.org/10.1090/S0002-9904-1948-09132-7}, \href
  {http://dx.doi.org/10.1090/S0002-9904-1948-09132-7}
  {\path{doi:10.1090/S0002-9904-1948-09132-7}}.

\bibitem{Rundell2017fractional}
W.~Rundell and Z.~Zhang.
\newblock Fractional diffusion: recovering the distributed fractional
  derivative from overposed data.
\newblock {\em Inverse Problems}, 33(3):035008, 27, 2017.
\newblock URL: \url{https://doi.org/10.1088/1361-6420/aa573e}.

\bibitem{sakamoto2011initial}
K.~Sakamoto and M.~Yamamoto.
\newblock Initial value/boundary value problems for fractional diffusion-wave
  equations and applications to some inverse problems.
\newblock {\em J. Math. Anal. Appl.}, 382(1):426--447, 2011.
\newblock URL: \url{http://dx.doi.org/10.1016/j.jmaa.2011.04.058}, \href
  {http://dx.doi.org/10.1016/j.jmaa.2011.04.058}
  {\path{doi:10.1016/j.jmaa.2011.04.058}}.

\bibitem{sakthivel2012approximate}
R.~Sakthivel, S.~Suganya, and S.~M. Anthoni.
\newblock Approximate controllability of fractional stochastic evolution
  equations.
\newblock {\em Computers \& Mathematics with Applications}, 63(3):660--668,
  2012.

\bibitem{tuan2017inverse}
N.~H. Tuan and E.~Nane.
\newblock Inverse source problem for time-fractional diffusion with discrete
  random noise.
\newblock {\em Statistics \& Probability Letters}, 120:126--134, 2017.

\bibitem{wharmby2013generalization}
A.~W. Wharmby and R.~L. Bagley.
\newblock Generalization of a theoretical basis for the application of
  fractional calculus to viscoelasticity.
\newblock {\em Journal of Rheology (1978-present)}, 57(5):1429--1440, 2013.

\bibitem{wharmby2014modifying}
A.~W. Wharmby and R.~L. Bagley.
\newblock Modifying maxwell's equations for dielectric materials based on
  techniques from viscoelasticity and concepts from fractional calculus.
\newblock {\em International Journal of Engineering Science}, 79:59--80, 2014.

\bibitem{zou2016galerkin}
G.-a. Zou.
\newblock A galerkin finite element method for time-fractional stochastic heat
  equation.
\newblock {\em arXiv preprint arXiv:1612.02082}, 2016.

\end{thebibliography}

\end{document}